\documentclass[10.5pt, a4paper]{amsart}
\usepackage{amssymb, amsmath, amscd, bm, amsthm, pdfpages, setspace, hyperref, mathtools, subcaption, graphicx, cite, xfrac, empheq, enumitem}

\def\@Rref#1{\hbox{\rm \ref{#1}}}
\def\Rref#1{\@Rref{#1}}
\theoremstyle{plain}
\newtheorem{theorem}{Theorem}[section]
\newtheorem{proposition}[theorem]{Proposition}

\newtheorem{assumption}[theorem]{Assumption}
\newtheorem{lemma}[theorem]{Lemma}
\newtheorem{corollary}[theorem]{Corollary}

\theoremstyle{definition}
\newtheorem{definition}{Definition}[section]
\newtheorem{example}[definition]{Example}
\newtheorem{remark}[definition]{Remark}

\newcommand{\sr}{\mathop{\rm r}}
\newcommand{\Ran}{\mathop{\rm Ran}\nolimits}
\newcommand{\Ker}{\mathop{\rm Ker}\nolimits}
\newcommand{\is}{\ensuremath{\mathrm{is}}}

\newcommand{\vertiii}[1]{{\vert\kern-0.25ex\vert\kern-0.25ex\vert #1 
		\vert\kern-0.25ex\vert\kern-0.25ex\vert}}
\newcommand{\biggvertiii}[1]{{\bigg\vert\kern-0.35ex\bigg\vert\kern-0.35ex\bigg\vert #1 
		\bigg\vert\kern-0.35ex\bigg\vert\kern-0.35ex\bigg\vert}}
	
\begin{document}

\title[Informativity for stabilization]{Data informativity for stabilization of discrete-time
	infinite-dimensional systems}

\thispagestyle{plain}

\author{Masashi Wakaiki}
\address{Graduate School of System Informatics, Kobe University, Nada, Kobe, Hyogo 657-8501, Japan}
 \email{wakaiki@ruby.kobe-u.ac.jp}
 \thanks{This work was supported in part by 
 	JSPS KAKENHI Grant Number 24K06866.}

\begin{abstract}
This paper develops a data-driven framework for
stabilization of 
discrete-time infinite-dimensional systems.
We investigate informativity for stabilization, 
defined as
the existence of a feedback gain that stabilizes
all systems
compatible with the available input-state data.
Assuming that infinite-length data are Bessel sequences, we first establish a sufficient condition 
for data informativity in the noise-free case.
We next show that
this sufficient condition is also 
necessary under 
a mild data assumption when
the input space is one-dimensional.
Furthermore, if the state sequence forms a frame,
then the sufficient condition can be extended to
the case of noisy data.
Finally, when the unstable part of the true system
is known to be finite-dimensional,
we derive a necessary and sufficient condition
for data informativity of finite-length data.
\end{abstract}

\keywords{Data-driven control,
	infinite-dimensional systems, stabilization.} 

\maketitle

\section{Introduction}
Model-based control is the standard paradigm 
for infinite-dimensional systems. It typically proceeds by first-principles modeling combined with
parameter identification; 
see, e.g., \cite{Demetriou1994,Orlov2000,Bin2016,Boiger2016,Mauroy2021,Chattopadhyay2025} and references therein for parameter-identification techniques for infinite-dimensional systems.
As an alternative to model-based control, 
data-driven control has rapidly matured in the finite-dimensional setting.
In this approach, system analysis and controller design are carried out directly from measured data, without constructing
mathematical models.
A powerful technique for data-driven control of infinite-dimensional systems is the Koopman approach \cite{Deutscher2025,Deutscher2025Automatica,
	Deutscher2025semilinear}, 
which relies on spectral analysis
of Koopman operators for 
partial differential equations overviewed in \cite{Nakao2020}.
Here we pursue a different route and work within the
informativity framework established for finite-dimensional systems in \cite{Waarde2020}.
Under this framework, we consider the set of
systems compatible with the available data. Then we 
examine whether every system
in the set satisfies the desired property, and 
design controllers achieving the control objective for 
all systems in the set. 
Our focus is on informativity for stabilization, i.e.,
the existence of a state feedback gain that stabilizes every 
system compatible with the input-state data.
This concept was introduced for finite-dimensional systems in \cite{Waarde2020}, where
necessary and 
sufficient conditions were also derived.
Subsequent work~\cite{Waarde2023SIAM,Bisoffi2024,Kaminaga2025} addressed informativity for stabilization in the presence of noise.
The informativity approach was also applied to
analysis of controllability \cite{Waarde2020} and
dissipativity~\cite{Waarde2022}, output regulation~\cite{Trentelman2021,Zhu2024},
model reduction~\cite{Burohman2023,
	Burohman2024, Ackermann2025}, and so on.
For a broader overview of data informativity, 
we refer the reader to the survey article \cite{Waarde2023}.

Let $X$ and $U$ be Hilbert spaces, and 
consider the following discrete-time infinite-dimensional system
with state space $X$ and input space $U$:
\[
x(k+1) = Ax(k) + Bu(k),\quad k \geq 0;\qquad 
x(0) = x_{\mathrm{ini}} \in X,
\]
where $A\colon X \to X$ and $B\colon U \to X$ are 
bounded linear operators.
Although $A$ and $B$ are unknown,
data on the state $x$ and the input $u$
are available.
A key distinction in the infinite-dimensional case is
that data of infinite length are required in general.
Such data sequences may be unbounded and therefore
cannot be handled
in the framework of bounded operators.
To circumvent this issue, we impose a
boundedness condition on the given data, i.e.,
we assume that
the data form Bessel sequences; 
see Definition \ref{def:Bessel} for the definition of
Bessel sequences.
This condition guarantees that
the synthesis operators associated with the data 
are bounded.
The synthesis operators generalize the matrices whose columns are the data vectors, and therefore play the same central role as the data matrices in the finite-dimensional setting.

We begin by deriving a sufficient condition for the given data
to be informative for stabilization.
When the input space $U$
is one-dimensional,
we show that 
this condition is also necessary under some mild
data assumption. Moreover, the 
structure of stabilizing gains is presented.
As in the finite-dimensional case, the 
condition is 
stated by using the range of the synthesis operator.
Unlike the range of a matrix, the range of 
an operator is not closed in general, which introduces technical difficulties.
To overcome the difficulties, we exploit the
prior information that the system operators $A$ and $B$ are bounded, together with standard results on orthogonal complements.
As a consequence,
the range of the synthesis operator associated with the state data need not coincide with $X$.
It suffices for the range to be dense in $X$.

A simple example shows that
the obtained condition 
is not robust to measurement noise.
Indeed, arbitrarily small noise causes the condition to fail;
see the first paragraph of Section~\ref{sec:noise}
for details.
Motivated by this observation, 
we investigate a sufficient condition for
informativity for stabilization in the presence of noise.
The key assumption in the infinite-dimensional setting 
is that the state sequence satisfies
a certain lower estimate and consequently
forms a frame for $X$;
see Definition~\ref{def:frame} for the definition of frames.
Frame theory has undergone extensive development, 
largely motivated by its applications in wavelet analysis; see, e.g., \cite[Chapter~5]{Mallat2008}.
Frames and Bessel sequences have also been employed
in the dynamical sampling problem, where
one studies whether families of operator iterates
$\{A^{k}x:x \in \Omega \text{~and~} 0\leq k \leq L(x)\}$ with $\Omega \subseteq X$ and $L(x) \in \mathbb{N}$
form frames; see, e.g., \cite{Aldroubi2017,Aldroubi2017JFA,
	Christensen2024}.
The frame-based condition on the data resembles
the necessary and sufficient condition 
for informativity for stabilization 
in the noise-free setting
for
finite-dimensional system
established in \cite{Waarde2020}.
However, the frame-based condition
is only sufficient in
the infinite-dimensional case.
To see how far it is from being necessary,
we show that 
data informativity together with an additional operator inequality
on
$\ell^2(\mathbb{N})$
implies the frame-based condition.

The results above are obtained under the assumption that data of infinite length
are available. In practice, however, 
we typically have only
finite-length data.
Therefore, we study informativity of finite-length data, assuming partial prior knowledge of the true system.
In particular, we consider the situation where
the subsystem to be stabilized is finite-dimensional and 
is not affected by the remaining infinite-dimensional subsystem.
One illustrative case involves a cascade system 
of the heat equation
with uncertain parameters
and an unknown finite-dimensional system.
We obtain necessary and sufficient conditions 
guaranteeing that 
the finite-length data are informative for stabilization.
These conditions can be checked using only
matrix computations
as in the finite-dimensional setting, since
we can focus on stabilization of the finite-dimensional
unstable subsystem.

This paper is organized as follows.
In Section~\ref{sec:system}, we 
introduce the system and the data we consider.
In Section~\ref{sec:identification},
we briefly discuss a data property for
system identification.
Section~\ref{sec:stabilization}
addresses data informativity
for stabilization in the noise-free case, while
Section~\ref{sec:noise} presents
a sufficient condition for data informativity
in the presence of noise.
In Section~\ref{sec:frame_based},
we derive two conditions equivalent to
this sufficient condition.
Section~\ref{sec:finite_data}
examines informativity of finite-length data
under the assumption that partial prior knowledge of the system structure is obtained. Finally,
Section~\ref{sec:conclusion}
concludes the paper.

\subsection*{Notation}
The set of positive integers 
is denoted by $\mathbb{N}$.
We define 
$\mathbb{N}_0 \coloneqq \mathbb{N} \cup \{0 \}$ and
$\mathbb{E}_r \coloneqq 
\{\lambda \in \mathbb{C}:|\lambda| >r  \}$
for $r >0$. The dimension of a vector space $V$ 
is denoted by $\dim V$.
The standard orthonormal basis in $\ell^2(\mathbb{N})$
is denoted by $(e_k)_{k \in \mathbb{N}}$,
where
for each $j,k \in \mathbb{N}$, the 
$j$th term of $e_k$ 
is $1$ if $j=k$ and $0$ if $j\neq k$.
Throughout this paper, all Hilbert spaces are assumed to be
complex. 
The inner product on a Hilbert space is denoted by
$\langle \cdot,\cdot \rangle$.
Let $X$ and $Y$ be Hilbert spaces.
The set of 
bounded linear operators from $X$ to $Y$
is denoted by $\mathcal{L}(X,Y)$.
We write $\mathcal{L}(X) \coloneqq \mathcal{L}(X,X)$.
For a linear operator $A\colon X \to Y$,
we denote by $D(A)$
the domain of $A$.
Let $T \in \mathcal{L}(X,Y)$. 
The range and the kernel of $T$
are denoted by
$\Ran T$ and $\Ker T$, respectively.
The Hilbert-space adjoint of $T$ is denoted by $T^*$.
Let $E$ be a subset of $X$.
The closure and 
the orthogonal complement of $E$ are denoted by
$\overline{E}$
and 
$E^{\perp}$, respectively.
We write $TE \coloneqq \{Tx: x\in E \}$ and
denote by $T|_E$ the restriction of $T$ to $E$.
Let $S \in \mathcal{L}(X)$.
The spectrum and the spectral radius of $S$
are denoted by 
$\sigma(S)$  and
$\sr(S)$, respectively.
For $\lambda \in \mathbb{C} \setminus \sigma(S)$,
we write $R(\lambda,S) \coloneqq (\lambda I - S)^{-1}$.
We write $S \geq 0$ if $S$ is self-adjoint and satisfies
$\langle Sx, x \rangle \geq 0$ for all $x \in X$.
For subspaces $X_1$ and $X_2$ of $X$ such that
$X_1 \cap X_2 = \{0 \}$, 
the direct sum of $X_1$ and $X_2$ is denoted by
$
X_1 \dotplus X_2 \coloneqq \{
x_1+x_2: x_1 \in X_1 \text{~and~} x_2 \in X_2
\}.
$

\section{Systems and data}
\label{sec:system}
Let $X$ and $U$ be 
non-trivial complex Hilbert spaces.
Throughout this paper,
$X \times U$ is
endowed with the standard 
inner product induced by 
the inner products on $X$ and $U$.
We consider a discrete-time infinite-dimensional system
\begin{equation}
	\label{eq:true_sys}
	x(k+1) = A_sx(k)+ B_su(k),\quad k \in \mathbb{N}_0;\qquad 
	x(0) = x_{\mathrm{ini}} \in X,
\end{equation}
where 
$x(k) \in X$, $u(k) \in U$,
$A_s \in \mathcal{L}(X)$, and $B_s \in \mathcal{L}(U,X)$.
We regard the operator pair $(A_s,B_s)$ as the true system.
The discrete-time system \eqref{eq:true_sys} appears 
when we study the following continuous-time system with zero-order hold.

\begin{example}
	\label{ex:sampled_data}
	Let $A_c$ be the generator of a $C_0$-semigroup $(T(t))_{t \geq 0}$
	on a Hilbert space $X$. We denote by $X_{-1}$ the extrapolation space of $X$
	associated with $(T(t))_{t \geq 0}$ and by $(T_{-1}(t))_{t \geq 0}$
	the $C_0$-semigroup on $X_{-1}$ that extends $(T(t))_{t \geq 0}$; see, e.g., \cite[Section~2.10]{Tucsnak2009}
	for the details of 
	$X_{-1}$ and $(T_{-1}(t))_{t \geq 0}$.
	Let $B_c \in \mathcal{L}(U,X_{-1})$. 
	By \cite[Lemma~2.2]{Logemann2003},
	the operator $S_r\colon U \to X$ defined by
	\[
	S_r v \coloneqq \int_0^{r} T_{-1}(t)B_c v dt,\quad v \in U
	\]
	satisfies $S_r \in \mathcal{L}(U,X)$ for all $r >0$.
	Consider the continuous-time system
	with zero-order hold of period $\tau >0$:
	\begin{equation}
		\label{eq:cont_time_system}
		\dot z(t) = A_cz(t) + B_cu(k),\quad k\tau \leq t < (k+1)\tau,~
		k \in \mathbb{N}_0;\qquad z(0) = x_{\mathrm{ini}} \in X.
	\end{equation}
	Then
	$x(k) \coloneqq z(k\tau)$, $k \in \mathbb{N}_0$,
	satisfies \eqref{eq:true_sys} with $A_s \coloneqq T(\tau)$ and 
	$B_s \coloneqq S_\tau$.
	Moreover, there exists a constant $M\geq 1$ such that 
	\[
	\|z(t)\| \leq M (\|x(k)\|+\|u(k)\|)
	\] 
	for all $x_{\mathrm{ini}} \in X$, 
	$t \in [k\tau,(k+1)\tau)$, and $k \in \mathbb{N}_0$.
	Hence,
	many stability properties of
	the discrete-time system~\eqref{eq:true_sys}
	are inherited by
	the continuous-time system~\eqref{eq:cont_time_system}.
	\hspace*{\fill} $\triangle$ 
\end{example}

Let 
$x_1 = (x_1(k))_{k \in \mathbb{N}}$ and
$x_0 = (x_0(k))_{k \in \mathbb{N}}$ be sequences in $X$. Let
$u_0 = (u_0(k))_{k \in \mathbb{N}}$ be a sequence in $U$.
We assume that the data $(x_1,x_0,u_0)$
are generated from the true system $(A_s,B_s)$.
Note that here we consider the data of infinite length
for control of infinite-dimensional systems,
which is different from the finite-dimensional setting
introduced in
\cite{Waarde2020}.
\begin{assumption}
	\label{assump:x1x0u0_relation}
	The data $(x_1,x_0,u_0)$ satisfy
	\begin{equation}
		\label{eq:x1x0u0_relation}
		x_1(k) = A_sx_0(k) + B_s u_0(k)
		\quad \text{for all $k \in \mathbb{N}$}.
	\end{equation}
\end{assumption}

Typical data satisfying Assumption~\ref{assump:x1x0u0_relation} 
are a single state-input trajectory over an infinite horizon 
and infinitely many state-input trajectories over finite horizons,
as shown in  Examples~\ref{ex:single_trajectory} 
and \ref{ex:inf_trajectories}, respectively.
\begin{example}
	\label{ex:single_trajectory}
	Let $(x(k))_{k \in \mathbb{N}_0}$ be 
	the solution of the evolution equation~\eqref{eq:true_sys}.
	If we define
	the data $(x_1,x_0,u_0)$ by
	\begin{align*}
		x_1(k) \coloneqq x(k), \quad 
		x_0(k) \coloneqq x(k-1),  \quad \text{and}\quad 
		u_0(k) \coloneqq u(k-1)
	\end{align*}
	for $k \in \mathbb{N}$, then
	\eqref{eq:x1x0u0_relation} holds.
	\hspace*{\fill} $\triangle$ 
\end{example}

\begin{example}
	\label{ex:inf_trajectories}
	For each $\ell \in \mathbb{N}$, let $k_{\ell} \in \mathbb{N}$
	and
	let $(x^{\ell}(k))_{k =0}^{k_{\ell}}$ be 
	the solution of the evolution equation~\eqref{eq:true_sys}
	for some initial state $x(0)=x_{\mathrm{ini}}^{\ell} \in X$ and input $u= (u^{\ell}(k))_{k =0}^{k_{\ell}-1}$.
	Define 
	$T_1 \coloneqq 0$ and
	$T_{\ell} \coloneqq \sum_{p=1}^{\ell-1} k_p$ for $\ell \geq 2$.
	If we define the data $(x_1,x_0,u_0)$ by
	\begin{align*}
		x_1(k+T_{\ell}) \coloneqq x^{\ell}(k), \quad 
		x_0(k+T_{\ell}) \coloneqq x^{\ell}(k-1), \quad \text{and}\quad 
		u_0(k+T_{\ell}) \coloneqq u^{\ell}(k-1)
	\end{align*}
	for $k=1,\dots,k_{\ell}$ and $\ell \in \mathbb{N}$, then
	\eqref{eq:x1x0u0_relation} holds.
	\hspace*{\fill} $\triangle$ 
\end{example}

We consider the situation where 
the true system $(A_s,B_s)$ is unknown but
the data $(x_1,x_0,u_0)$ are available.
Given the data $(x_1,x_0,u_0)$, we define
the set $\Sigma_{\is}$ of systems by
\begin{equation}
	\label{eq:def_Sigma_is}
	\Sigma_{\is} \coloneqq  \{
	(A,B) \in \mathcal{L}(X) \times \mathcal{L}(U,X):
	x_1(k) = Ax_0(k)+ Bu_0(k) \text{~for all $k \in \mathbb{N}$}
	\}.
\end{equation}
By definition, systems in $\Sigma_{\is}$ are 
compatible with the data
$(x_1,x_0,u_0)$.
Under Assumption~\ref{assump:x1x0u0_relation},
we have $(A_s,B_s) \in \Sigma_{\is}$.

Bessel sequences play a fundamental role in our approach 
to data-driven stabilization for infinite-dimensional systems;
see \cite[Section~3.2]{Christensen2016} for basic 
properties of Bessel sequences.
\begin{definition}
	\label{def:Bessel}
	A sequence $(\eta_k)_{k \in \mathbb{N}}$ in a Hilbert space $Y$
	is called a {\em Bessel sequence} if there exists a constant $\alpha >0$
	such that for all $y \in Y$,
	\begin{equation}
		\label{eq:Bessel}
		\sum_{k=1}^{\infty} | \langle y,\eta_k\rangle  |^2 \leq \alpha \|y\|^2.
	\end{equation}
\end{definition}

For a Bessel sequence $(\eta_k)_{k \in \mathbb{N}}$ in a Hilbert space $Y$,
we define the operator $T$ on $Y$ by
\[
Ty \coloneqq ( \langle y,\eta_k\rangle )_{k \in \mathbb{N}},\quad 
y \in Y.
\]
By the inequality \eqref{eq:Bessel}, we obtain
$T \in \mathcal{L}
(Y,\ell^2(\mathbb{N}))$.
The adjoint of $T$ is given by
\begin{equation}
	\label{eq:synthesis}
	T^*w = \sum_{k=1}^{\infty} w_k \eta_k,\quad 
	w = (w_k)_{k \in \mathbb{N}} \in \ell^2(\mathbb{N}).
\end{equation}
We call $T^*$ the {\em synthesis operator associated with $(\eta_k)_{k \in \mathbb{N}}$}.
Conversely, if the operator given in \eqref{eq:synthesis} 
belongs to $\mathcal{L}(\ell^2(\mathbb{N}),Y)$, then
the associated sequence $(\eta_k)_{k \in \mathbb{N}}$ 
is a Bessel sequence; see \cite[Theorem~3.2.3]{Christensen2016}.

In this paper, we assume that all data form Bessel sequences.
Intuitively, this assumption guarantees 
that the data neither diverge 
nor become overly 
concentrated in any particular direction.
\begin{assumption}
	\label{assump:Bessel}
	Each component of the data $(x_1,x_0,u_0)$
	is a Bessel sequence, 
	i.e., $x_1$ and
	$x_0$ are Bessel sequences in $X$, and
	$u_0$ is a Bessel sequence in $U$.
\end{assumption}

Under Assumption~\ref{assump:Bessel},
we denote by $\Xi_1$, $\Xi_0$, and $\Upsilon_0$
the synthesis operators associated with
$x_1$,
$x_0$, and 
$u_0$, respectively. 
We conclude this section by characterizing 
the system set $\Sigma_{\is}$ in terms of 
the synthesis operators
$(\Xi_1,\Xi_0,\Upsilon_0)$.
\begin{lemma}
	\label{lem:Xi1_rep}
	Suppose that  
	the data $(x_1,x_0,u_0)$
	satisfy Assumption~\ref{assump:Bessel}, and let 
	$(\Xi_1,\Xi_0,\Upsilon_0)$ be the synthesis operators 
	associated with $(x_1,x_0,u_0)$.
	Then the set $\Sigma_{\is}$ defined by \eqref{eq:def_Sigma_is}
	satisfies
	\begin{equation}
		\label{eq:Sigma_is_characterization}
		\Sigma_{\is} = \{
		(A,B) \in \mathcal{L}(X) \times \mathcal{L}(U,X):
		\Xi_1 = 
		A\Xi_0 + B \Upsilon_0
		\}.
	\end{equation}
\end{lemma}
\begin{proof}
	We denote by $\Sigma_{\mathrm{syn}}$
	the set on the right-hand side of \eqref{eq:Sigma_is_characterization}.
	First, we prove the inclusion $\Sigma_{\is} \subseteq
	\Sigma_{\mathrm{syn}}$.
	Let $(A,B) \in \Sigma_{\is}$.
	By the definition \eqref{eq:def_Sigma_is} 
	of $\Sigma_{\is}$, we have
	\[
	x_1(k) = Ax_0(k)+ Bu_0(k)
	\]
	for all $k \in \mathbb{N}$.
	Combining this with
	the expression \eqref{eq:synthesis} of synthesis operators, we obtain
	\begin{align*}
		\Xi_1 w 
		=
		\sum_{k=1}^{\infty} w_k \big(Ax_{0}(k)+Bu_{0}(k) \big) 
		=
		(A\Xi_0 + B \Upsilon_0)w
	\end{align*}
	for all $w= (w_k)_{k \in \mathbb{N}} \in \ell^2(\mathbb{N})$.
	Therefore, $(A,B) \in \Sigma_{\mathrm{syn}}$.
	
	Next, we prove the inclusion $\Sigma_{\mathrm{syn}} \subseteq
	\Sigma_{\is} $.
	If $(A,B) \in \Sigma_{\mathrm{syn}}$, then
	the $k$th unit vector $e_k$ in $\ell^2(\mathbb{N})$
	satisfies
	\[
	x_1(k) = 
	\Xi_1 e_k = (A\Xi_0 + B \Upsilon_0)e_k = 
	Ax_0(k) + Bu_0(k)
	\]
	for each  $k \in \mathbb{N}$.
	Thus,
	$(A,B) \in \Sigma_{\is}$.
\end{proof}

\section{Informativity for system identification}
\label{sec:identification}
In this section, we briefly discuss a data property for identifying
the true system. 
As in the finite-dimensional case, this identification property 
can be characterized in terms of the range of synthesis operators. 
However, the significant difference is that the closure of the range is used for infinite-dimensional systems.
This will also present a challenge in extending the results on
data-driven stabilization from finite-dimensional systems in Section~\ref{sec:stabilization}.

The following notion 
captures  the data property that identifies the true system $(A_s,B_s)$.
\begin{definition}
	Under Assumption~\ref{assump:x1x0u0_relation},
	the data $(x_1,x_0,u_0)$ are called
	{\em informative for system identification} if 
	$\Sigma_{\is} = \{(A_s,B_s)\}$.
\end{definition}

The next proposition provides
a necessary and sufficient condition for 
the data to be informative for system 
identification, which
is the infinite-dimensional version of
\cite[Proposition~6]{Waarde2020}.
The combination of 
the dense range of synthesis operators and the prior knowledge that the operators $(A,B) \in \Sigma_{\is}$ are bounded
leads to system identification.
\begin{proposition}
	\label{prop:identification}
	Suppose that  
	the data $(x_1,x_0,u_0)$
	satisfy Assumptions~\ref{assump:x1x0u0_relation} and 
	\ref{assump:Bessel}. Let
	$(\Xi_1,\Xi_0,\Upsilon_0)$ be the synthesis operators 
	associated with $(x_1,x_0,u_0)$. Then
	$(x_1,x_0,u_0)$ are
	informative for system identification
	if and only if 
	\begin{equation}
		\label{eq:identification}
		\overline{
			\Ran
			\begin{bmatrix}
				\Xi_0 \\ \Upsilon_0
			\end{bmatrix}
		} =
		X \times U.
	\end{equation}
\end{proposition}

\begin{proof}
	Define $H \in \mathcal{L}(\ell^2(\mathbb{N}), X \times U) $
	by
	\[
	H  \coloneqq \begin{bmatrix}
		\Xi_0 \\ \Upsilon_0
	\end{bmatrix}.
	\]
	First, suppose that $\overline{\Ran H } = X \times U$. 	Since $\Ker H ^* = (\overline{\Ran H })^{\perp}
	= \{ 0\}$, it follows that 
	$H H ^*$ is injective and hence 
	has an algebraic inverse 
	$(H H ^*)^{-1}$ on $X\times U$
	with domain $D((H H ^*)^{-1}) \coloneqq \Ran (HH^*)$. 
	Let $(A,B) \in \Sigma_{\is}$. Then
	Lemma~\ref{lem:Xi1_rep} yields
	\begin{equation}
		\label{eq:AB_expression}
		\Xi_1 H ^*(HH^*)^{-1}
		\psi =	
		\begin{bmatrix}
			A & B
		\end{bmatrix} \psi	
		\quad \text{for all~$\psi \in \Ran (HH^*) $}.
	\end{equation}
	Since
	$
	\overline{\Ran (HH^*)} = X\times U 
	$	
	and since 
	$\begin{bmatrix}
		A & B
	\end{bmatrix}$ is a bounded operator
	on $X \times U$,
	the system $(A,B)$ satisfying
	\eqref{eq:AB_expression} is uniquely determined.
	This yields $A= A_s$ and $B=B_s$
	under Assumption~\ref{assump:x1x0u0_relation}.
	Thus, the data $(x_1,x_0,u_0)$ are
	informative for system identification.
	
	Conversely, suppose that 
	the data $(x_1,x_0,u_0)$ are
	informative for system identification.
	Assume, to get a contradiction, that
	$\overline{\Ran H } \neq  X \times U$. 
	Since $(\overline{\Ran H })^{\perp} \neq  \{0\} $,
	we can take
	\[
	\begin{bmatrix}
		\xi_0 \\ \upsilon_0
	\end{bmatrix} \in (\overline{\Ran H })^{\perp} \setminus \{ 0\}.
	\]
	Let $\zeta \in X \setminus \{0 \}$. 
	Define the operators 
	$A\in \mathcal{L}(X)$ and $B\in \mathcal{L}(U,X)$ by
	\[
	Ax \coloneqq A_sx+
	\langle x,\xi_0 \rangle\zeta,\quad x \in X;
	\quad \text{and} \quad  
	Bu \coloneqq B_su + 
	\langle u,\upsilon_0\rangle  \zeta,\quad u \in U.
	\]
	Then 
	$(A,B) \neq  (A_s,B_s)$. 
	Using  Assumption~\ref{assump:x1x0u0_relation}
	and Lemma~\ref{lem:Xi1_rep},
	we obtain 
	\begin{equation}
		\label{eq:Xi1_ABH}
		\Xi_1 = \begin{bmatrix}
			A_s & B_s	
		\end{bmatrix} H.
	\end{equation}
	Since
	\[
	\begin{bmatrix}
		A-A_s & B - B_s	
	\end{bmatrix} 
	\begin{bmatrix}
		x \\ u
	\end{bmatrix} = 
	\left\langle
	\begin{bmatrix}
		x \\ u
	\end{bmatrix},
	\begin{bmatrix}
		\xi_0 \\ \upsilon_0
	\end{bmatrix}
	\right\rangle \zeta = 
	0\qquad\text{for all~$\begin{bmatrix}
			x \\ u
		\end{bmatrix} \in \Ran H$},
	\]
	it follows that 
	\begin{equation}
		\label{eq:ABH_AX_BU}
		\begin{bmatrix}
			A_s & B_s	
		\end{bmatrix} H = A\Xi_0  + B\Upsilon_0.
	\end{equation}
	By \eqref{eq:Xi1_ABH} and \eqref{eq:ABH_AX_BU},
	$\Xi_1 = A\Xi_0+B\Upsilon_0$.
	Using Lemma~\ref{lem:Xi1_rep} again, we obtain
	$(A,B) \in \Sigma_{\is}$. This contradicts
	informativity for system identification.
\end{proof}

\section{Informativity for stabilization}
\label{sec:stabilization}
In this section, we investigate a data property for 
stabilization by state feedback.
First, we present the main results in this section.
We then give the proof of these results.
Finally, we make a comment on controller design and
discuss conditions used in the results.
\subsection{Main results}
For $\gamma \in (0,1)$,
a bounded linear operator $F$ 
on a Hilbert space
is said to be 
{\em power stable with decay rate $\gamma$}
if there exists a constant $M\geq 1$ such that 
$\|F^k\| \leq M \gamma^k$ for all $k \in \mathbb{N}_0$.
Given an operator
$K \in \mathcal{L}(X,U)$ and constants
$M\geq 1$ and $\gamma > 0$, 
we define the set $\Sigma_{K,M,\gamma}$ of systems
by
\[
\Sigma_{K,M,\gamma} \coloneqq \{
(A,B) \in \mathcal{L}(X) \times \mathcal{L}(U,X):
\|(A+BK)^k\| \leq M \gamma^k \text{~for all $k \in \mathbb{N}_0$}
\}.
\]
The following notion of 
informativity for stabilization was introduced for 
finite-dimensional systems in
\cite[Definition~12]{Waarde2020}.
This data property implies that
every system in the set $\Sigma_{\is}$ defined by
\eqref{eq:def_Sigma_is} is stabilized 
by a single feedback gain.
\begin{definition}
	\label{def:IS}
	For a constant $\gamma \in (0,1)$,
	the data $(x_1,x_0,u_0)$ are called
	{\em 
		informative for stabilization with decay rate $\gamma$}
	if there exist an operator 
	$K \in \mathcal{L}(X,U)$ and a constant
	$M\geq 1$ 
	such that $\Sigma_{\is} \subseteq \Sigma_{K,M,\gamma}$.
	An operator $K\in \mathcal{L}(X,U)$ is called
	a {\em stabilizing gain with decay rate $\gamma$ for $(x_1,x_0,u_0)$} if
	$\Sigma_{\is} \subseteq \Sigma_{K,M,\gamma}$ 
	for some constant $M \geq 1$.
\end{definition}

The implication (ii) $\Rightarrow $ (i) of
the following theorem gives
a sufficient condition for the data to be 
informative for stabilization.
\begin{theorem}
	\label{thm:informative_stabilization}
	Suppose that  
	the data $(x_1,x_0,u_0)$
	satisfy Assumptions~\ref{assump:x1x0u0_relation} and 
	\ref{assump:Bessel}. Let
	$(\Xi_1,\Xi_0,\Upsilon_0)$ be the synthesis operators 
	associated with $(x_1,x_0,u_0)$.
	Then the following statements are equivalent for a fixed $\gamma \in (0,1)$:
	\begin{enumerate}
		[label=\upshape(\roman*), leftmargin=*, widest=ii]
		\item 
		The data $(x_1,x_0,u_0)$  are
		informative for stabilization with decay rate $\gamma$, and 
		there exists a stabilizing gain $K \in \mathcal{L}(X,U)$ 
		with decay rate $\gamma$ for $(x_1,x_0,u_0)$
		such that
		\begin{equation}
			\label{eq:im_Xi0_inclusion}
			\Ran 
			\begin{bmatrix}
				\Xi_0 \\ K \Xi_0
			\end{bmatrix}
			\subseteq
			\Ran 
			\begin{bmatrix}
				\Xi_0 \\ \Upsilon_0
			\end{bmatrix}.		
		\end{equation}
		
		\item $\overline{\Ran \Xi_0} = X$, and there exists 
		a densely-defined operator 
		$\Xi_0^{\dagger} \colon X \to \ell^2(\mathbb{N})$ with
		domain $D(\Xi_0^{\dagger}) \coloneqq \Ran \Xi_0$
		such that 
		\begin{enumerate}[label=\upshape\alph*), leftmargin=*, widest=b]
			\item $\Xi_0 \Xi_0^{\dagger}x = x$ for all $x \in D(\Xi_0^{\dagger})$;
			\item $\Upsilon_0 \Xi_0^{\dagger}$ can be extended to
			an operator in $\mathcal{L}(X,U)$; and
			\item $\Xi_1 \Xi_0^{\dagger}$ can be extended to
			an operator  in $\mathcal{L}(X)$, and 
			this extension is power stable with decay
			rate $\gamma$.
		\end{enumerate}
	\end{enumerate}
\end{theorem}

In the following corollary of Theorem~\ref{thm:informative_stabilization},
we omit the condition~\eqref{eq:im_Xi0_inclusion} 
from statement~(i)
in the case $\dim U = 1$, and consequently,
informativity for stabilization is equivalent to
statement~(ii).
Note that a sequence $u_0$ in $\mathbb{C}$ is a Bessel
sequence if and only if $u_0 \in \ell^2(\mathbb{N})$.
\begin{corollary}
	\label{coro:informative_stabilization_SI}
	Assume that  $\dim U = 1$.
	Suppose that  
	the data $(x_1,x_0,u_0)$
	satisfy Assumptions~\ref{assump:x1x0u0_relation} and 
	\ref{assump:Bessel}. Let
	$(\Xi_1,\Xi_0,\Upsilon_0)$ be the synthesis operators 
	associated with $(x_1,x_0,u_0)$.
	Assume either that $\Ker \Xi_0 \setminus \{u_0\}^{\perp} \neq 
	\emptyset$ or
	that the data $(x_1,x_0,u_0)$ are not
	informative for system identification.
	For a fixed $\gamma \in (0,1)$,
	the data $(x_1,x_0,u_0)$  are
	informative for stabilization with decay rate $\gamma$
	if and only if statement \textup{(ii)} of Theorem~\ref{thm:informative_stabilization}
	holds.
\end{corollary}

The condition $\Ker \Xi_0 \setminus \{u_0\}^{\perp} \neq 
\emptyset$ in Corollary~\ref{coro:informative_stabilization_SI} is satisfied, for example, when
$x_0(k_1) = c x_0(k_2)$ and 
$u_0(k_1) \neq  cu_0(k_2)$ for some $k_1,k_2 \in \mathbb{N}$
and $c \in \mathbb{R}$. 
In Example~\ref{ex:inf_trajectories}, such data can be obtained by applying different inputs to the same initial state.
\begin{remark}[Difference from the finite-dimensional case]
	In the finite-dimensional setting \cite[Section~IV]{Waarde2020},
	the following set $\Sigma_K$ of systems
	was introduced for a fixed
	$K \in \mathcal{L}(X,U)$:
	\begin{align*}
		\Sigma_{K} \coloneqq \{
		(A,B) \in \mathbb{R}^{n \times n} \times \mathbb{R}^{n \times m}:
		&~\text{There exist $M\geq 1$ and $0 \leq \gamma < 1$ such that} \\
		&~\text{$\|(A+BK)^k\| \leq M \gamma^k$ for all $k \in \mathbb{N}_0$}
		\}.
	\end{align*}
	The major difference from $\Sigma_{K,M,\gamma}$
	is that the constants $M$ and $\gamma$ are allowed to depend on $(A,B)$.
	However, it was implicitly shown in \cite[Theorem~16]{Waarde2020} that
	$\Sigma_{\is} \subseteq \Sigma_{K}$ if and only if
	$\Sigma_{\is} \subseteq \Sigma_{K,M,\gamma}$ for some
	constants $M\geq 1$ and $\gamma \in (0,1)$.
	The spectral radius is continuous 
	for matrices, 
	but this property does not hold for general bounded operators; see 
	Kakutani's example in the solution of \cite[Problem~104]{Halmos1982}.
	This is the
	reason why we use $\Sigma_{K,M,\gamma}$
	in the infinite-dimensional setting. 
	The details will be
	provided in the proof of Lemma~\ref{lem:ker_lemma}.
	\hspace*{\fill} $\triangle$ 
\end{remark}

\subsection{Proof of Theorem~\ref{thm:informative_stabilization}}
For the proof of the implication (i) $\Rightarrow $ (ii) and Corollary~\ref{coro:informative_stabilization_SI},
we need a preliminary lemma.
This lemma is the infinite-dimensional version of \cite[Lemma~15]{Waarde2020}.
As in Proposition~\ref{prop:identification}, we 
use the closure of 
the range of synthesis operators, which is different from the finite-dimensional case.
\begin{lemma}
	\label{lem:ker_lemma}
	Suppose that  
	the data $(x_1,x_0,u_0)$
	satisfy Assumptions~\ref{assump:x1x0u0_relation}  and \ref{assump:Bessel}. Let
	$(\Xi_1,\Xi_0,\Upsilon_0)$ be the synthesis operators 
	associated with $(x_1,x_0,u_0)$.
	Assume that
	$K \in \mathcal{L}(X,U)$ satisfies
	$\Sigma_{\is} \subseteq
	\Sigma_{K,M,\gamma}$ for some 
	constants $M\geq 1$ and $\gamma > 0$. Then
	\begin{equation}
		\label{eq:Xi_Up_ran_inclusion}
		\Ran \begin{bmatrix}
			I \\ K
		\end{bmatrix} \subseteq
		\overline{\Ran \begin{bmatrix}
				\Xi_0 \\ \Upsilon_0
		\end{bmatrix}}.
	\end{equation}
\end{lemma}
\begin{proof}
	Given the data $(x_0,u_0)$, we 
	define 
	the set $\Sigma_{\is}^0$ of systems by
	\[
	\Sigma_{\is}^0 \coloneqq  \{
	(A,B) \in \mathcal{L}(X) \times \mathcal{L}(U,X):
	\text{$0 =  Ax_0(k)+Bu_0(k)$ for all
		$k \in \mathbb{N}$}
	\}.
	\]
	We will show that 
	\begin{equation}
		\label{eq:spe_radius}
		\sr(A_0+B_0K) =0\quad \text{for all $(A_0,B_0)
			\in \Sigma_{\is}^0$}.
	\end{equation}
	Let $(A,B) \in \Sigma_{\is}$ and $(A_0,B_0)
	\in \Sigma_{\is}^0$.
	Define $F \coloneqq A+BK$ and $F_0 \coloneqq A_0+B_0K$.
	Then 
	for all $n \in \mathbb{N}$, we obtain 
	$(A+nA_0, B+nB_0) \in \Sigma_{\is}$
	and 
	\[
	F + n F_0 = (A+nA_0) + (B+nB_0)K.
	\]
	Therefore, by assumption,
	\begin{equation}
		\label{eq:FnF0_power}
		\|(F + n F_0)^k\| \leq M \gamma^k
	\end{equation}
	for all $n \in \mathbb{N}$ and $k \in \mathbb{N}_0$.
	
	Assume, to reach a contradiction, that $\sr(F_0)>0$.
	Then there exists $n_0 \in \mathbb{N}$ such that
	$\sr(F_0) > \gamma/n_0.$
	For $n \in \mathbb{N}$, define $G_n \in \mathcal{L}(X)$ by
	\[
	G_n \coloneqq \frac{F}{n} +F_0.
	\]
	Since \eqref{eq:FnF0_power} yields
	\[
	\|G_n^k\| \leq M \left( \frac{\gamma}{n}\right)^k
	\]
	for all $n \in \mathbb{N}$ and $k \in \mathbb{N}_0$,
	it follows from
	\cite[Theorem~2.1]{Lubich1991} 
	that for all $n \in \mathbb{N}$,
	\[
	\|R(\lambda, G_n)\| \leq \frac{M}{|\lambda| - \gamma/n}
	\]
	whenever $|\lambda | > \gamma/n$; see also
	\cite[Proposition~II.1.10]{Eisner2010}.
	Note that the constant $M$ does not 
	depend on $n$.
	For all $\lambda \in \mathbb{E}_{\sr(F_0)}$ and 
	$n \geq n_0$,
	\[
	R(\lambda,G_n) - R(\lambda,F_0) =
	\frac{R(\lambda,G_n) FR(\lambda,F_0)}{n},
	\]
	and hence
	\begin{equation}
		\label{eq:Rel_F0_conv}
		R(\lambda,F_0) = \lim_{n \to \infty}
		R(\lambda,G_n).
	\end{equation}
	By Vitali's theorem
	(see, e.g., \cite[Theorem~A.5]{Arendt2001}),
	there exists a holomorphic function $f\colon
	\mathbb{E}_{\gamma/n_0} \to X$ such that 
	\begin{equation}
		\label{eq:f_conv}
		f(\lambda) = \lim_{n \to \infty}
		R(\lambda,G_n) 
	\end{equation}
	for all $\lambda \in \mathbb{E}_{\gamma/n_0}$.
	There
	exists $\lambda_{\max} \in \sigma(F_0)$ such that 
	$|\lambda_{\max}| = \sr(F_0)$.
	Take a sequence $(\lambda_k)_{k \in \mathbb{N}}$
	in $\mathbb{E}_{\sr(F_0)}$ such that 
	$\lambda_k \to \lambda_{\max}$ as $k\to \infty$.
	By \eqref{eq:Rel_F0_conv} and 
	\eqref{eq:f_conv}, we obtain
	\begin{equation}
		\label{eq:f_lambdan_conv}
		\|f(\lambda_k)\| = 
		\|R(\lambda_k,F_0)\| \to \infty\quad 
		(k \to \infty);
	\end{equation}
	see, e.g., \cite[Proposition~IV.1.3]{Engel2000}.
	On the other hand, $\lambda_{\max} \in \mathbb{E}_{\gamma/n_0}$, and hence
	\[
	\lim_{k \to \infty}
	\|f(\lambda_k)\| = \|f(\lambda_{\max})\| < \infty,
	\]
	which contradicts \eqref{eq:f_lambdan_conv}.
	Thus, \eqref{eq:spe_radius} is obtained.
	
	Let $(A_0,B_0)
	\in \Sigma_{\is}^0$. Then
	\[
	\big((A_0+B_0K)^*A_0,(A_0+B_0K)^*B_0 \big)
	\in \Sigma_{\is}^0.
	\]
	Since $(A_0+B_0K)^*(A_0+B_0K)$ is a self-adjoint operator,
	it follows from \eqref{eq:spe_radius} that
	\[
	\|(A_0+B_0K)^*(A_0+B_0K)\| = 
	\sr\big((A_0+B_0K)^*(A_0+B_0K)\big) = 0.
	\]
	Therefore, $(A_0,B_0)
	\in \Sigma_{\is}^0$ implies that $A_0+B_0K = 0$.
	
	The desired 
	inclusion \eqref{eq:Xi_Up_ran_inclusion} is equivalent to
	\begin{equation}
		\label{eq:Xi_Up_ker_inc}
		\Ker 
		\begin{bmatrix}
			\Xi_0^* & \Upsilon_0^*
		\end{bmatrix} \subseteq
		\Ker 
		\begin{bmatrix}
			I & K^*
		\end{bmatrix}.
	\end{equation}
	To show the inclusion \eqref{eq:Xi_Up_ker_inc},
	let 
	\[
	\begin{bmatrix}
		\xi_0 \\ \upsilon_0
	\end{bmatrix} \in \Ker
	\begin{bmatrix}
		\Xi_0^* & \Upsilon_0^*
	\end{bmatrix}.
	\]
	Take $\zeta \in X \setminus \{0 \}$
	arbitrarily.
	Define 
	$A_0 \in \mathcal{L}(X)$ and 
	$B_0 \in \mathcal{L}(U,X)$ by
	$
	A_0 x \coloneqq \langle x,\xi_0 \rangle \zeta
	$ for $x\in X$
	and  
	$B_0 u \coloneqq \langle u,\upsilon_0 \rangle \zeta$
	for $u \in U$.
	Then
	\[
	A_0\Xi_0w + B_0 \Upsilon_0w =
	\langle \Xi_0w,\xi_0\rangle \zeta + \langle \Upsilon_0 w, \upsilon_0\rangle \zeta =
	\langle w, \Xi_0^* \xi_0 + \Upsilon_0^* \upsilon_0\rangle 
	\zeta = 0
	\]
	for all $w \in \ell^2(\mathbb{N})$.
	Hence, $(A_0,B_0) \in \Sigma_{\is}^0$ by Lemma~\ref{lem:Xi1_rep}
	with $x_1(k) \equiv 0$.
	Since $A_0+B_0K = 0$ by the above argument,
	we obtain
	\[
	0 = 
	(A_0+B_0K)x =
	\langle x,\xi_0\rangle \zeta + \langle Kx,\upsilon_0\rangle \zeta =
	\langle x,\xi_0+K^*\upsilon_0\rangle \zeta
	\]
	for all $x \in X$.
	This yields
	$\xi_0+K^*\upsilon_0 = 0$, i.e.,
	\[
	\begin{bmatrix}
		\xi_0 \\ \upsilon_0
	\end{bmatrix} \in \Ker
	\begin{bmatrix}
		I & K^*
	\end{bmatrix}.
	\]
	Thus, the inclusion~\eqref{eq:Xi_Up_ker_inc} holds.
\end{proof}

We now give the proof of Theorem~\ref{thm:informative_stabilization}.
As in the case of system identification studied in Proposition~\ref{prop:identification},
we employ the density argument together with
the boundedness of the operators $A$, $B$, and $K$.
For simplicity of notation, we define 
$G  \in \mathcal{L}(X, X \times U) $ and 
$H \in \mathcal{L}(\ell^2(\mathbb{N}), X \times U) $
by
\begin{equation}
	\label{eq:GH_def}
	G \coloneqq \begin{bmatrix}
		I \\ K
	\end{bmatrix} \quad \text{and} \quad 
	H \coloneqq \begin{bmatrix}
		\Xi_0 \\ \Upsilon_0
	\end{bmatrix}
\end{equation}
for $K \in \mathcal{L}(X,U)$,
$\Xi_0 \in \mathcal{L}(\ell^2(\mathbb{N}), X)$, and 
$\Upsilon_0 \in \mathcal{L}(\ell^2(\mathbb{N}), U)$.
\begin{proof}[Proof of Theorem~\ref{thm:informative_stabilization}]
	First, we prove the implication (ii) $\Rightarrow$ (i).
	By Lemma~\ref{lem:Xi1_rep} and condition~a),
	\begin{equation}
		\label{eq:Xi1Xi0_dagger}
		\Xi_1 \Xi_0^{\dagger}x = 
		(A\Xi_0 + B\Upsilon_0)  \Xi_0^{\dagger}x = 
		(A+B\Upsilon_0\Xi_0^{\dagger}) x
	\end{equation}
	for all $(A,B) \in \Sigma_{\is}$ and 
	$x \in D(\Xi_0^{\dagger})$.
	Let $K \in \mathcal{L}(X,U)$
	be the extension of
	$\Upsilon_0 \Xi_0^{\dagger}\colon
	D(\Xi_0^{\dagger}) \to U$,
	where condition~b) was used.
	By \eqref{eq:Xi1Xi0_dagger}
	and the density of $D(\Xi_0^{\dagger}) $
	in $X$,
	condition~c) implies that there exists $M \geq 1$ such that
	$
	\|(A+BK)^k\| \leq M \gamma^k
	$
	for all $(A,B) \in \Sigma_{\is}$ and $k \in \mathbb{N}$.
	Since $\Sigma_{\is} \subseteq \Sigma_{K,M,\gamma}$,
	it follows that 
	the data $(x_1,x_0,u_0)$  are
	informative for stabilization with decay rate $\gamma$.
	By condition~a) and the definition of $K$,
	we have 
	\[
	\begin{bmatrix}
		I \\ K
	\end{bmatrix} x = 
	\begin{bmatrix}
		\Xi_0 \\ \Upsilon_0
	\end{bmatrix} \Xi_0^{\dagger}x \in \Ran \begin{bmatrix}
		\Xi_0 \\ \Upsilon_0
	\end{bmatrix}
	\]
	for all $x  \in D(\Xi_0^{\dagger})= \Ran (\Xi_0)$. Hence,
	\eqref{eq:im_Xi0_inclusion} is satisfied.
	
	Next, we prove the implication (i) $\Rightarrow$ (ii).
	By informativity for stabilization and Lemma~\ref{lem:ker_lemma}, we obtain
	$\overline{\Ran \Xi_0} =X$.
	Let $G  \in \mathcal{L}(X, X \times U) $ and 
	$H \in \mathcal{L}(\ell^2(\mathbb{N}), X \times U) $
	be defined as in \eqref{eq:GH_def}, where
	$K \in \mathcal{L}(X,U)$ is a stabilizing gain 
	with decay rate $\gamma$ for $(x_1,x_0,u_0)$
	satisfying \eqref{eq:im_Xi0_inclusion}.
	Define $H_0 \coloneqq H|_{(\Ker H)^{\perp}}$.
	Then $H_0$ is injective and has an algebraic inverse $H_0^{-1} \colon X\times U \to (\Ker H)^{\perp}$
	with domain $D(H_0^{-1}) \coloneqq \Ran H$. 
	For all $x \in D(H_0^{-1})$, we obtain 
	\begin{equation}
		\label{eq:HH_0}
		H H_0^{-1}x = H_0H_0^{-1}x = x.
	\end{equation}
	Noting that
	\eqref{eq:im_Xi0_inclusion} gives
	$
	Gx \in
	\Ran 
	H = D(H_0^{-1})
	$
	for all $x \in \Ran \Xi_0$,
	we define the densely-defined operator $\Xi \colon X \to \ell^2(\mathbb{N})$ by
	\begin{align*}
		\Xi x&\coloneqq H_0^{-1} Gx,\quad x \in D(\Xi)
		\coloneqq \Ran \Xi_0.
	\end{align*}
	By \eqref{eq:HH_0}, 
	for all $x \in D(\Xi)$,
	\[
	\begin{bmatrix}
		x \\ Kx
	\end{bmatrix}=
	G x = H
	\Xi  x =
	\begin{bmatrix}
		\Xi_0\Xi x \\ \Upsilon_0\Xi x
	\end{bmatrix}.
	\]
	This implies that 
	condition~a)  is satisfied.
	By the density of $D(\Xi) = \Ran \Xi_0$ in $X$,
	condition~b) is also satisfied.
	Finally, let $(A,B) \in \Sigma_{\is}$. Since 
	Lemma~\ref{lem:Xi1_rep} shows that
	\[
	(A+BK)x =
	(A\Xi_0 + B\Upsilon_0)\Xi  x = 
	\Xi_1 \Xi x
	\]
	for all $x \in D(\Xi)$,
	the operator 
	$\Xi_1\Xi$ can be extended to
	$A+BK \in \mathcal{L}(X)$
	by the density of $D(\Xi)$ in $X$.
	Thus, condition~c) is satisfied
	by the power stability of $A+BK$.
\end{proof}

\subsection{Proof of Corollary~\ref{coro:informative_stabilization_SI}
	and  comments}
We base the proof of Corollary~\ref{coro:informative_stabilization_SI}
on the following lemma on the inclusion~\eqref{eq:im_Xi0_inclusion}.

\begin{lemma}
	\label{lem:suff_cond_for_inclusion}
	Let
	$\Xi_0 \in \mathcal{L}(\ell^2(\mathbb{N}), X)$, 
	$\Upsilon_0 \in \mathcal{L}(\ell^2(\mathbb{N}), U)$, and 
	$K \in \mathcal{L}(X,U)$. Define
	$G  \in \mathcal{L}(X, X \times U) $ and 
	$H \in \mathcal{L}(\ell^2(\mathbb{N}), X \times U) $
	by \eqref{eq:GH_def}.
	Then
	the following statements hold:
	\begin{enumerate}[label=\upshape\alph*), leftmargin=*, widest=b]
		\item $\Ran (G\Xi_0) \subseteq \Ran H$  if and only if
		$
		\Ran(K \Xi_0  - \Upsilon_0 ) \subseteq \Upsilon_0(\Ker \Xi_0)
		$.
		\item 
		Suppose that 
		$\Ran (G\Xi_0) \subseteq \overline{\Ran H}$.
		If $\dim \Ran \Upsilon_0 < \infty$ and 
		$\Upsilon_0(\Ker \Xi_0) = \Ran \Upsilon_0$, then 
		$\Ran (G\Xi_0) \subseteq \Ran H$.
		In particular, if $\dim U = 1$ and 
		$\Ker \Xi_0 \nsubseteq \Ker \Upsilon_0$, then $\Ran (G\Xi_0) \subseteq \Ran H$.
	\end{enumerate}
\end{lemma}

\begin{proof}
	a) 
	Suppose that $\Ran (G\Xi_0) \subseteq \Ran H$.
	Let $w \in \ell^2(\mathbb{N})$.
	By assumption,
	$
	\Xi_0 w = \Xi_0 \mu
	$ and 
	$
	K \Xi_0 w = \Upsilon_0 \mu
	$
	for some $\mu \in \ell^2(\mathbb{N})$.
	Since $\Xi_0 (\mu-w) = 0$, we have 
	$\mu-w \in \Ker \Xi_0$.
	This yields
	\[
	K\Xi_0 w - \Upsilon_0 w =
	\Upsilon_0 (\mu -w) \in \Upsilon_0 (\Ker \Xi_0).
	\]
	
	Conversely,
	suppose that $\Ran
	(K \Xi_0  - \Upsilon_0 ) \subseteq \Upsilon_0(\Ker \Xi_0)
	$.  Let $w \in \ell^2(\mathbb{N})$.
	There exists $\mu \in \Ker \Xi_0$ such that 
	$
	K\Xi_0 w - \Upsilon_0 w = \Upsilon_0 \mu.
	$
	Then we obtain
	$\Xi_0 w = \Xi_0 (w+\mu) $ and 
	$K\Xi_0 w = \Upsilon_0 (w+\mu)$.
	This implies that $G\Xi_0 w \in \Ran H$.
	
	b) Let $w \in \ell^2(\mathbb{N})$.
	By assumption, we have
	$
	G \Xi_0 w \in 
	\overline{
		\Ran 
		H
	}.
	$
	There exists a sequence 
	$(\mu_n)_{n \in \mathbb{N}}$ in 
	$\ell^2(\mathbb{N})$
	such that 
	\[
	\begin{bmatrix}
		I \\ K
	\end{bmatrix} \Xi_0 w =
	\lim_{n \to \infty}
	\begin{bmatrix}
		\Xi_0 \\ \Upsilon_0
	\end{bmatrix} \mu_n.
	\]
	It follows from $\dim \Ran \Upsilon_0 < \infty$ that
	$\Ran \Upsilon_0$ is closed. Hence,
	$K\Xi_0w \in \Ran \Upsilon_0$, which together with
	$\Upsilon_0(\Ker \Xi_0) = \Ran \Upsilon_0$ 
	yields
	\[
	K \Xi_0 w - \Upsilon_0 w \in \Ran \Upsilon_0 = \Upsilon_0(\Ker \Xi_0).
	\]
	Thus, we obtain 
	$\Ran (G\Xi_0) \subseteq \Ran H$  by statement~a).
	In particular, 
	if $\dim U = 1$ and 
	$\Ker \Xi_0 \nsubseteq \Ker \Upsilon_0$, then
	$\Upsilon_0(\Ker \Xi_0) = \Ran \Upsilon_0$, and hence
	$\Ran (G\Xi_0) \subseteq \Ran H$. 
\end{proof}

Now we are able to give the following proof of 
Corollary~\ref{coro:informative_stabilization_SI}.
\begin{proof}[Proof of Corollary~\ref{coro:informative_stabilization_SI}]
	Assume that 
	$\Ker \Xi_0 \setminus \{u_0\}^{\perp} \neq 
	\emptyset$ or that
	the data $(x_1,x_0,u_0)$ are not
	informative for system identification.
	By Theorem~\ref{thm:informative_stabilization},
	it is enough to show that 
	every stabilizing gain $K \in \mathcal{L}(X,U)$
	with decay rate $\gamma$
	for
	$(x_1,x_0,u_0)$ satisfies the inclusion~\eqref{eq:im_Xi0_inclusion}.
	
	Let $K \in \mathcal{L}(X,U)$ be a stabilizing gain 
	with decay rate $\gamma$
	for
	$(x_1,x_0,u_0)$.
	Let $\xi \in X$ and $\upsilon \in U$. 
	Define 	$G  \in \mathcal{L}(X, X \times U) $ and 
	$H \in \mathcal{L}(\ell^2(\mathbb{N}), X \times U) $ by
	\eqref{eq:GH_def}.
	Then
	\[
	\left\langle 
	Gx, \begin{bmatrix}
		\xi \\ \upsilon 
	\end{bmatrix}
	\right\rangle =
	\langle \xi,x\rangle  + \langle \upsilon,Kx\rangle  = 0
	\]
	for all $x \in X$ if and only if $\xi+K^*\upsilon = 0$.
	This implies that
	$(\Ran G)^{\perp}$ is given by
	\[
	(\Ran G)^{\perp} = 
	\left\{
	\begin{bmatrix}
		-K^*\upsilon \\ \upsilon
	\end{bmatrix}:
	\upsilon \in U
	\right\},
	\]
	and hence $\dim\,(\Ran G)^{\perp} = 1$.
	By Lemma~\ref{lem:ker_lemma},
	we obtain
	$\Ran G \subseteq \overline{\Ran H}$ and hence
	$(\overline{\Ran H })^{\perp} \subseteq (\Ran G)^{\perp}$.
	This yields 
	\[
	(\overline{\Ran H })^{\perp} = \{ 0\}
	\quad \text{or}
	\quad
	(\overline{\Ran H })^{\perp} = (\Ran G)^{\perp}.
	\]

	First, we suppose that  $(\overline{\Ran H })^{\perp} = \{ 0\}$. Then
	$
	\overline{\Ran H} = X\times U
	$, and
	Proposition~\ref{prop:identification} shows that
	$(x_1,x_0,u_0)$ are
	informative for system identification. Hence,
	$\Ker \Xi_0 \setminus \{u_0\}^{\perp} \neq 
	\emptyset$ by assumption.
	Since $\Ker \Upsilon_0 =  \{u_0\}^{\perp}$,
	we obtain $\Ker \Xi_0 \nsubseteq \Ker \Upsilon_0$.
	Hence,
	\eqref{eq:im_Xi0_inclusion}
	follows from Lemma~\ref{lem:suff_cond_for_inclusion}.b).
	
	Next, we suppose that
	$(\overline{\Ran H })^{\perp} = (\Ran G)^{\perp}$. 
	Since $\Ran G$ is closed, 
	we obtain
	\[
	\overline{\Ran H } = \big((\overline{\Ran H })^{\perp} \big)^{\perp} =
	\big( (\Ran G)^{\perp}\big)^{\perp} = \Ran G.
	\]
	Therefore,
	for all $w \in \ell^2(\mathbb{N})$, 
	there exists $x \in X$ such that 
	$x = \Xi_0 w$ and $Kx = \Upsilon_0 w$.
	This yields $K \Xi_0 = \Upsilon_0$, and thus
	\eqref{eq:im_Xi0_inclusion} holds.
\end{proof}

From the proof of 
Theorem~\ref{thm:informative_stabilization} and 
Corollary~\ref{coro:informative_stabilization_SI},
we also obtain the structure of stabilizing
gains as shown in the following 
remark.
\begin{remark}[Structure of stabilizing gains]
	\label{rem:structure_gain}
	Let the assumptions of 
	Corollary~\ref{coro:informative_stabilization_SI} be satisfied, and let $\gamma \in (0,1)$.
	If statement~(ii) of Theorem~\ref{thm:informative_stabilization} holds, then
	the extension of $\Upsilon_0\Xi_0^{\dagger}$ is a 
	stabilizing gain with decay rate $\gamma$ for $(x_1,x_0,u_0)$.
	Conversely, every stabilizing
	gain with decay rate $\gamma$ for $(x_1,x_0,u_0)$
	is the extension of
	$\Upsilon_0\Xi_0^{\dagger}$ for some right inverse $\Xi_0^{\dagger}$ 
	of $\Xi_0$ satisfying
	the properties given in statement~(ii) of Theorem~\ref{thm:informative_stabilization}.
	\hspace*{\fill} $\triangle$ 
\end{remark}

To omit the condition~\eqref{eq:im_Xi0_inclusion},
we instead assume in Corollary~\ref{coro:informative_stabilization_SI} that 
$\Ker \Xi_0 \setminus \{u_0\}^{\perp} \neq 
\emptyset$ or that 
$(x_1,x_0,u_0)$ are not
informative for system identification.
One may ask whether informativity for stabilization alone
implies the condition~\eqref{eq:im_Xi0_inclusion} without this assumption.
The following example shows that the answer to this question is negative.

\begin{example}
	Let
	$X = \ell^2(\mathbb{N})$ and $U = \mathbb{C}$.
	Take
	$x_1(k) = 0$,
	$x_0(k) = e_k/k$, and 
	$u_0(k) = 1/k^{3/2}$
	for $k \in \mathbb{N}$.
	Then $x_1$, $x_0$, and 
	$u_0 $ are Bessel sequences.
	Let $\Xi_0$ and $\Upsilon_0$ be the 
	synthesis operators of $x_0$ and $u_0$, respectively.
	We will show that the following statements hold:
	\begin{enumerate}[label=\upshape\alph*), leftmargin=*, widest=b]
		\item 
		$\Ker \Xi_0 \setminus \{u_0\}^{\perp} =
		\emptyset$, and
		$(x_1,x_0,u_0)$ are
		informative for system identification.
		\item $(x_1,x_0,u_0)$ are
		informative for stabilization 
		with all decay rates in $(0,1)$.
		\item 
		There does not exist $K \in \mathcal{L}(X,U)$
		such that \eqref{eq:im_Xi0_inclusion} holds.
	\end{enumerate}

	Since $\Ker \Xi_0 = \{0\}$, we obtain
	$\Ker \Xi_0 \setminus \{u_0\}^{\perp} =
	\emptyset$.
	To show that $(x_1,x_0,u_0)$ are
	informative for system identification, 
	it suffices to prove that \eqref{eq:identification} holds,
	which is equivalent to
	\begin{equation}
		\label{eq:kernel_XiUp}
		\Ker \begin{bmatrix}
			\Xi_0^* & \Upsilon_0^*
		\end{bmatrix} = \{0 \}.
	\end{equation}
	Let $\xi = (\xi_n)_{n \in \mathbb{N}}\in X$ and $\upsilon \in U$.
	For all $n \in \mathbb{N}$,
	the $n$th element of 
	$
	\Xi_0^*\xi + \Upsilon_0^*\upsilon
	$
	can be written as
	\[
	\frac{\xi_n}{n} + \frac{\upsilon}{n^{3/2}}.
	\]
	Therefore, if $\Xi_0^*\xi + \Upsilon_0^*\upsilon= 0$, then
	$
	\xi_n = - \upsilon/n^{1/2}
	$
	for all $n \in \mathbb{N}$. It follows that $\xi \notin X$
	if $\upsilon \neq 0$.
	Hence,
	$\Xi_0^*\xi + \Upsilon_0^*\upsilon = 0$ implies that
	$\xi = 0$ and $\upsilon = 0$. Since \eqref{eq:kernel_XiUp} holds, we conclude that 
	statement~a) holds.
	
	Since 	$x_1(k) = 0$
	for all $k \in \mathbb{N}$, we have
	$(0,0) \in \Sigma_{\is}$.
	It follows from informativity for system identification that
	$\Sigma_{\is} = \{(0,0)\}$. Hence, 
	$
	\Sigma_{\is} 
	\subseteq \Sigma_{K,M,\gamma}$ is satisfied for 
	all 
	$K \in \mathcal{L}(X,U)$, $M\geq 1$, and $\gamma \in (0,1)$.
	This implies that 
	statement~b) holds.

	If some $K \in \mathcal{L}(X,U)$
	satisfies \eqref{eq:im_Xi0_inclusion}, 
	then $K\Xi_0 = \Upsilon_0$ by 
	Lemma~\ref{lem:suff_cond_for_inclusion}.a)
	and 
	$\Upsilon_0(\Ker \Xi_0) =  \{ 0\}$.
	Hence, for all $\xi = (\xi_n)_{n \in \mathbb{N}} \in \Ran \Xi_0$,
	\[
	K\xi = 
	\sum_{n=1}^{\infty} \frac{\xi_n}{n^{1/2}}. 
	\]
	This contradicts $K \in \mathcal{L}(X,U)$.
	Thus, statement~c) holds.
	\hspace*{\fill} $\triangle$ 
\end{example}

\section{Informativity for stabilization in the presence of noise}
\label{sec:noise}
The range condition $\overline{\Ran \Xi_0} = X$ in statement~(ii) of Theorem~\ref{thm:informative_stabilization}
is not robust to noise in data, as the following simple example
shows.
Let $X = \ell^2(\mathbb{N})$ and let
$x_0(k) = e_k/k$ for $k \in \mathbb{N}$.
Then 
the synthesis operator $\Xi_0$ of the Bessel sequence $(x_0(k))_{k \in \mathbb{N}}$
satisfies 
$\overline{\Ran \Xi_0} = X$.
Fix $k_0 \in \mathbb{N}$ and introduce the noise $\delta_0 \in X$ 
defined by
\[
\delta_0(k) \coloneqq 
\begin{cases}
	0, & k\neq k_0, \\
	-x_0(k_0),& k = k_0.
\end{cases}
\]
Then the synthesis operator  associated with
the noisy data $x_0+\delta_0$ does not 
satisfy the range condition.
Since $\|\delta_0\| = 1/k_0$, 
arbitrarily small noise can destroy the range condition 
by choosing $k_0 \in \mathbb{N}$ sufficiently large.

In this section, we derive a sufficient condition 
for data informativity that is robust to noise.
Let 
$\widetilde x_1 = (\widetilde x_1(k))_{k \in \mathbb{N}}$,
$\widetilde x_0 = (\widetilde x_0(k))_{k \in \mathbb{N}}$,
$\delta_1 = (\delta_1(k))_{k \in \mathbb{N}}$, and
$\delta_0 = (\delta_0(k))_{k \in \mathbb{N}}$ be sequences in $X$. Let
$\widetilde u_0 = (\widetilde u_0(k))_{k \in \mathbb{N}}$ and
$\theta_0 = (\theta_0(k))_{k \in \mathbb{N}}$ be sequences in $U$.
The sequences $(\widetilde x_1,\widetilde x_0,\widetilde u_0)$
are the noisy data measured instead of the noise-free
data $(x_1,x_0,u_0)$ considered in Section~\ref{sec:stabilization}.
The sequences $(\delta_1,\delta_0,\theta_0)$ are the noise
signals included in
$(\widetilde x_1,\widetilde x_0,\widetilde u_0)$.
Here we consider the situation where only the noisy data
$(\widetilde x_1, \widetilde x_0, \widetilde u_0)$ are available.

As in Assumption~\ref{assump:Bessel},
we make the following assumption on 
$(\widetilde x_1, \widetilde x_0, \widetilde u_0)$ and 
$(\delta_1,\delta_0,\theta_0)$.
\begin{assumption}
	\label{assump:Bessel_noise}
	Each component of 
	the noisy data $(\widetilde x_1, \widetilde x_0, \widetilde u_0)$ and
	the noise $(\delta_1,\delta_0,\theta_0)$
	is a Bessel sequence, 
	i.e., 
	$\widetilde x_1$,
	$\widetilde x_0$,
	$\delta_1$, and
	$\delta_0$ are Bessel sequences in $X$, and
	$\widetilde u_0$ and
	$\theta_0$ are Bessel sequences in $U$.
\end{assumption}

Under Assumption~\ref{assump:Bessel_noise},
the synthesis operators 
associated with $(\widetilde x_1, \widetilde x_0, \widetilde u_0)$
and $(\delta_1,\delta_0,\theta_0)$
are denoted by
$(\widetilde \Xi_1, \widetilde \Xi_0, \widetilde \Upsilon_0)$ and $(\Delta_1, \Delta_0, \Theta_0)$, respectively.
Given the noisy data 
$(\widetilde x_1, \widetilde x_0, \widetilde u_0)$, 
we consider the following inequalities on $
\ell^2(\mathbb{N})$ for the noise
$(\delta_1,\delta_0,\theta_0)$:
\begin{equation}
	\label{eq:noise_cond1}
	\Delta_1 \Omega
	\Omega^* \Delta_1^* \leq 
	c_1^2  \widetilde \Xi_1 \Omega \Omega^*\widetilde \Xi_1^*
\end{equation}
and
\begin{equation}
	\label{eq:noise_cond2}
	\begin{bmatrix}
		\Delta_0 \\ 
		\Theta_0 
	\end{bmatrix}
	\Omega
	\Omega^* \begin{bmatrix}
		\Delta_0^* &
		\Theta_0^* 
	\end{bmatrix}
	\leq
	c_0^2  \begin{bmatrix}
		\widetilde \Xi_0 \\ 
		\widetilde \Upsilon_0
	\end{bmatrix}
	\Omega
	\Omega^* \begin{bmatrix}
		\widetilde \Xi_0^* &
		\widetilde \Upsilon_0^* 
	\end{bmatrix}
\end{equation}
for some constants $c_1,c_0 \geq 0$ and some
operator $\Omega \in \mathcal{L}(X,\ell^2(\mathbb{N}))$.
These conditions imply that the noise $(\delta_1,\delta_0,\theta_0)$ is relatively small compared to the 
measured data $(\widetilde x_1, \widetilde x_0, \widetilde u_0)$.
Such noise classes were also considered in 
the finite-dimensional setting \cite{Waarde2023SIAM,Bisoffi2024,Kaminaga2025}.
The reason for using inequalities of the above form
is technical and will be made clear in the proof of Theorem~\ref{thm:informative_stabilization_noise}.

Suppose that 
Assumption~\ref{assump:Bessel_noise} holds. We define the class $\mathcal{N}(c_1,c_0,\Omega)$
of noise by
\[
\mathcal{N}(c_1,c_0,\Omega)
\coloneqq 
\{
(\delta_1,\delta_0,\theta_0):
\text{Conditions~\eqref{eq:noise_cond1} and 
	\eqref{eq:noise_cond2} hold}
\}
\]
for constants
$c_1,c_0 \geq 0$ and 
operator $\Omega \in \mathcal{L}(X,\ell^2(\mathbb{N}))$.
Given $(\widetilde x_1, \widetilde x_0, \widetilde u_0)$, 
the set $\widetilde \Sigma_{\is}(c_1,c_0,\Omega)$ of
systems is defined by
\begin{align*}
	&\widetilde \Sigma_{\is}(c_1,c_0,\Omega) \coloneqq 
	\big\{
	(A,B) \in \mathcal{L}(X) \times \mathcal{L}(U,X):
	\text{There exist $(\delta_1,\delta_0,\theta_0) \in \mathcal{N}(c_1,c_0,\Omega)$} \\
	&\qquad \text{such that~} \widetilde x_1(k) - 
	\delta_1(k) = A\big(\widetilde x_0(k)
	- 
	\delta_0(k) \big)+ B
	\big(\widetilde u_0(k) -\theta_0(k)\big)~\text{for all $k \in \mathbb{N}$}
	\big\}
\end{align*}
for 
$c_1,c_0 \geq 0$ and 
$\Omega \in \mathcal{L}(X,\ell^2(\mathbb{N}))$.
By construction, $\widetilde \Sigma_{\is}(c_1,c_0,\Omega)$ is the set of all systems
compatible with the noise-free data
$(\widetilde x_1 - \delta_1,\widetilde x_0 - \delta_0,\widetilde u_0 - \theta_0)$.
If the noise that is actually
added to the data $(x_1,x_0,u_0)$ satisfying 
Assumption~\ref{assump:x1x0u0_relation}
belongs to $\mathcal{N}(c_1,c_0,\Omega)$, then 
$\widetilde \Sigma_{\is}(c_1,c_0,\Omega)$
contains the true system $(A_s,B_s)$.
As $c_1$ and $c_0$ increase, 
$\widetilde \Sigma_{\is}(c_1,c_0,\Omega)$ becomes larger.
It is worth mentioning that 
$\delta_1$ and $\theta_0$ can also be
regarded as external disturbances
to the system.

We examine the following 
notion of informativity for stabilization
in the presence of noise.
\begin{definition}
	Let $\gamma \in (0,1)$, 
	$c_1,c_0 \geq 0$, and 
	$\Omega \in \mathcal{L}(X,\ell^2(\mathbb{N}))$.
	The noisy data $(\widetilde x_1, \widetilde x_0, \widetilde u_0)$ are called
	{\em 
		informative for stabilization  with decay rate $\gamma$ under the noise class
		$\mathcal{N}(c_1,c_0,\Omega)$}
	if there exist an operator 
	$K \in \mathcal{L}(X,U)$ and a constant
	$M\geq 1$ 
	such that $	\widetilde \Sigma_{\is}(c_1,c_0,\Omega) 
	\subseteq \Sigma_{K,M,\gamma}$.	
\end{definition}

To discuss informativity for
stabilization in the presence of  noise,
we introduce the notion of frames for
Hilbert spaces. 
Concerning the general theory of
frames,
we refer to
\cite{Christensen2016}
and \cite[Chapter~5]{Mallat2008}.
\begin{definition}
	\label{def:frame}
	A sequence $(\eta_k)_{k \in \mathbb{N}}$ in a Hilbert space $Y$
	is called a {\em frame} for $Y$ if there exist constants 
	$\alpha, \beta >0$
	such that for all $y \in Y$,
	\[
	\beta \|y\|^2 \leq \sum_{k=1}^{\infty} | \langle 
	y,\eta_k\rangle  |^2 \leq \alpha \|y\|^2.
	\]
\end{definition}
The synthesis operator associated with a frame 
is surjective and hence has a right inverse in
$\mathcal{L}(Y,\ell^2(\mathbb{N}))$; see e.g.,
\cite[Lemma~5.1.5]{Christensen2016}.
Bessel sequences do not 
have these properties in general.

The following lemma, called Douglas' lemma~\cite[Theorem~1]{Douglas1966}, relates  
range inclusion, majorization, and
factorization for operators on Hilbert spaces, which is a key tool for
studying data informativity based on frames;
see also 
\cite[Proposition~12.1.2]{Tucsnak2009} and
\cite[Theorem~A.3.65]{Curtain2020}.
\begin{lemma}[Douglas' lemma]
	\label{lem:Douglas}
	Let $X$, $Y$, and $Z$ be Hilbert spaces. Let $A \in \mathcal{L}(Y,X)$ and 
	$B \in \mathcal{L}(Z,X)$.
	Then the following statements are equivalent:
	\begin{enumerate}[label=\upshape(\roman*), leftmargin=*, widest=iii]
		\item $\Ran A \subseteq \Ran B$.
		\item $AA^* \leq c^2 BB^*$ for some $c \geq 0$.
		\item There exists an operator $C \in \mathcal{L}(Y,Z)$ such that 
		$A=BC$.
	\end{enumerate}
	Moreover, 
	the operator $C \in \mathcal{L}(Y,Z)$ in statement~\textup{(iii)}
	satisfies $\|C\| \leq c$, where 
	the constant $c\geq 0$ is as in statement~\textup{(ii)}.
\end{lemma}

The following theorem gives a sufficient condition for 
the noisy data to be informative for stabilization.
\begin{theorem}
	\label{thm:informative_stabilization_noise}
	Suppose that  
	the noisy data $(\widetilde x_1,
	\widetilde x_0, \widetilde u_0)$ and the noise 
	$(\delta_1,\delta_0,\theta_0)$
	satisfy Assumption~\ref{assump:Bessel_noise}. Let 
	$(\widetilde\Xi_1,\widetilde\Xi_0,\widetilde\Upsilon_0)$ 
	be the synthesis operators 
	associated with $(\widetilde x_1,
	\widetilde x_0, \widetilde u_0)$.
	Assume that 
	$\widetilde x_0$
	is a frame for $X$ and that 
	there exist a right inverse $\widetilde \Xi_0^{\dagger} \in \mathcal{L}(X,
	\ell^2(\mathbb{N}))$ of $\widetilde \Xi_0$ and constants 
	$M\geq 1$ and $\gamma  > 0$ 
	such that 
	\begin{equation}
		\label{eq:til_Xi_power_stability}
		\|(\widetilde \Xi_1 \widetilde \Xi_0^{\dagger})^k\| \leq M \gamma^k
	\end{equation}
	for all $k \in \mathbb{N}_0$.
	Then 
	$K \coloneqq \widetilde \Upsilon_0
	\widetilde \Xi_0^{\dagger} \in \mathcal{L}(X,U)$ satisfies 
	$\widetilde \Sigma_{\is}(c_1,c_0,\widetilde \Xi_0^{\dagger}) 
	\subseteq \Sigma_{K,M,\widetilde \gamma}$
	for all $c_1 \geq 0$ and $0\leq c_0 < 1/M$, where
	\[
	\widetilde \gamma \coloneqq \frac{1+Mc_1}{1-Mc_0} \gamma >0.
	\]
	In particular,
	if $c_1,c_0\geq 0$ satisfy
	\begin{equation}
		\label{eq:c1c0_cond}
		\gamma c_1 + c_0 < \frac{1-\gamma}{M},
	\end{equation}
	then  $(\widetilde x_1, \widetilde x_0, \widetilde u_0)$ are 
	informative for stabilization  
	with decay rate $\widetilde \gamma \in (0,1)$ under the noise class
	$\mathcal{N}(c_1,c_0,\widetilde \Xi_0^{\dagger})$.

\end{theorem}
\begin{proof}
	Let
	$c_1 \geq 0$ and $0\leq c_0 < 1/M$. 
	If $\widetilde\Sigma_{\is}(c_1,c_0,\widetilde \Xi_0^{\dagger}) =\emptyset$, then the result is trivial, so we assume that
	$\widetilde\Sigma_{\is}(c_1,c_0,\widetilde \Xi_0^{\dagger}) \neq \emptyset$.
	Take $(A,B) \in \widetilde \Sigma_{\is}(c_1,c_0,\widetilde \Xi_0^{\dagger}) $ arbitrarily.
	Since $\widetilde \Xi_0\widetilde \Xi_0^{\dagger} =I$, we have
	\begin{equation}
		\label{eq:ABK_rep_noise}
		A+BK = (A\widetilde \Xi_0 + B\widetilde \Upsilon_0) \widetilde \Xi_0^{\dagger}.
	\end{equation}
	By the definition of $\widetilde \Sigma_{\is}(c_1,c_0,\widetilde \Xi_0^{\dagger})$ and 
	Lemma~\ref{lem:Xi1_rep}, 
	there exist $(\delta_1,\delta_0,\theta_0) \in \mathcal{N}(c_1,c_0,\widetilde \Xi_0^{\dagger})$ such that
	the synthesis operators $(\Delta_1, \Delta_0, \Theta_0)$
	of $(\delta_1,\delta_0,\theta_0)$ satisfy
	\begin{equation}
		\label{eq:AxiBUp_rep}
		A\widetilde \Xi_0 + B\widetilde \Upsilon_0 = 
		(\widetilde \Xi_1 - \Delta_1) - 
		(A\Delta_0+B \Theta_0).
	\end{equation}
	By
	Lemma~\ref{lem:Douglas} (Douglas' lemma) and
	\eqref{eq:noise_cond1},
	there exists $\Phi_1 \in \mathcal{L}(X)$
	with $\|\Phi_1\| \leq c_1$
	such that 
	\[
	\widetilde \Xi_1 \widetilde \Xi_0^{\dagger}\Phi_1
	=
	\Delta_1 \widetilde \Xi_0^{\dagger}.
	\] 
	Using \eqref{eq:noise_cond2}, we also have 
	\[
	\begin{bmatrix}
		\widetilde \Xi_0 \\ \widetilde \Upsilon_0 
	\end{bmatrix} \widetilde \Xi_0^{\dagger}
	\Phi_0 =
	\begin{bmatrix}
		\Delta_0 \\ \Theta_0 
	\end{bmatrix} \widetilde \Xi_0^{\dagger}
	\]
	for some $\Phi_0 \in \mathcal{L}(X)$
	satisfying  $\|\Phi_0\| \leq c_0$.
	Hence,
	\[
	\big(
	(\widetilde \Xi_1 - \Delta_1) - 
	(A\Delta_0+B \Theta_0)\big) \widetilde \Xi_0^{\dagger} = 
	\widetilde \Xi_1\widetilde \Xi_0^{\dagger}(I+\Phi_1) 
	+
	(A\widetilde \Xi_0 + B\widetilde \Upsilon_0) \widetilde \Xi_0^{\dagger}
	\Phi_0.
	\]
	Combining this with \eqref{eq:ABK_rep_noise} and 
	\eqref{eq:AxiBUp_rep},
	we derive
	\begin{equation}
		\label{eq:ABK_Xi1Xi0}
		(A+BK) (I - \Phi_0) =  
		\widetilde \Xi_1\widetilde \Xi_0^{\dagger} (I+\Phi_1).
	\end{equation}
	
	Define $\vertiii{\cdot} \colon X \to \mathbb{R}$ by
	\[
	\vertiii{x} \coloneqq 
	\sup_{k \in \mathbb{N}_0}
	\|\gamma^{-k}
	(\widetilde \Xi_1\widetilde \Xi_0^{\dagger})^{k}x\|,
	\quad 
	x \in X.
	\]
	Then $\vertiii{\cdot}$ is a norm on $X$.
	The same symbol $\vertiii{\cdot}$
	will be used for 
	the induced operator norm.
	By \eqref{eq:til_Xi_power_stability},
	\begin{align}
		\label{eq:vertiii_prop1}
		\|x\| \leq \vertiii{x} \leq M \|x\|
	\end{align}
	for all $x \in X$,
	and
	\begin{equation}
		\label{eq:vertiii_prop2}
		\vertiii{\widetilde \Xi_1\widetilde \Xi_0^{\dagger} }
		\leq \gamma.
	\end{equation}
	Since \eqref{eq:vertiii_prop1} yields 
	$\vertiii{\Phi_1}  \leq 
	Mc_1$,
	we obtain
	\begin{equation}
		\label{eq:I-Phi1}
		\vertiii{I + \Phi_1} \leq 1 + Mc_1.
	\end{equation}
	Similarly, $\vertiii{\Phi_0} \leq Mc_0 < 1$.
	Hence, $I - \Phi_0$ is invertible in 
	$\mathcal{L}(X)$, and
	\begin{equation}
		\label{eq:I-Phi0_inv}
		\vertiii{(I - \Phi_0)^{-1}} \leq \frac{1}{1-Mc_0}.
	\end{equation}
	Combining \eqref{eq:ABK_Xi1Xi0}--\eqref{eq:I-Phi0_inv},
	we derive
	\[
	\vertiii{A+BK}
	\leq 
	\vertiii{\widetilde \Xi_1\widetilde \Xi_0^{\dagger}} \,
	\vertiii{I+\Phi_1} \,
	\vertiii{(I-\Phi_0)^{-1}}
	\leq \frac{1+Mc_1}{1-Mc_0} \gamma = \widetilde \gamma.
	\]
	This and \eqref{eq:vertiii_prop1} give
	\[
	\|(A+BK)^k\| \leq M\vertiii{(A+BK)^k} \leq M \widetilde \gamma^k
	\]
	for all $k \in \mathbb{N}_0$, and therefore
	$(A,B) \in \Sigma_{K,M,\widetilde \gamma}$.
	This proves the first assertion.
	Since the condition~\eqref{eq:c1c0_cond}
	is equivalent to
	$\widetilde \gamma < 1$,
	the second assertion is also proved.
\end{proof}

\section{Frame-based sufficient condition and its equivalent conditions}
\label{sec:frame_based}
In this section, we return to the noise-free case and investigate the 
frame-based sufficient condition obtained in Theorem~\ref{thm:informative_stabilization_noise}.
We present its equivalent conditions in terms of
informativity for stabilization and operator inequalities.

The following theorem shows that 
the frame-based  condition (statement (ii))
is equivalent to the existence of a stabilizing gain that 
satisfies a certain operator inequality.
\begin{theorem}
	\label{thm:equiv_frame1}
	Suppose that  
	the data $(x_1,x_0,u_0)$
	satisfy Assumptions~\ref{assump:x1x0u0_relation} and 
	\ref{assump:Bessel}. Let
	$(\Xi_1,\Xi_0,\Upsilon_0)$ be the synthesis operators 
	associated with $(x_1,x_0,u_0)$.
	Then the following statements are equivalent
	for a fixed $\gamma \in (0,1)$:
	\begin{enumerate}[label=\upshape(\roman*), leftmargin=*, widest=ii]
		\item 
		The data $(x_1,x_0,u_0)$ are
		informative for stabilization with decay rate $\gamma$, 
		and there exist 
		a constant $c \geq 0$ and 
		a stabilizing gain $K \in \mathcal{L}(X,U)$ 
		with decay rate $\gamma$ for $(x_1,x_0,u_0)$
		such that 
		the following inequality on  $
		\ell^2(\mathbb{N})$ holds:
		\begin{equation}
			\label{eq:operator_ineq}
			(\Xi_0 +  K^*\Upsilon_0)^*(\Xi_0 +  K^*\Upsilon_0) 
			\leq c^2 
			(\Xi_0^* \Xi_0  +   \Upsilon_0^* \Upsilon_0)^2.
		\end{equation}
		
		\item The sequence $x_0$
		is a frame for $X$, and there exist a right inverse $\Xi_0^{\dagger} \in \mathcal{L}(X,
		\ell^2(\mathbb{N}))$ of $\Xi_0$ such that 
		$\Xi_1  \Xi_0^{\dagger}$
		is power stable with decay rate $\gamma$.
	\end{enumerate}
\end{theorem}

\begin{proof}
	Let $G  \in \mathcal{L}(X, X \times U) $ and 
	$H \in \mathcal{L}(\ell^2(\mathbb{N}), X \times U) $
	be defined as in \eqref{eq:GH_def}.
	First, we prove the implication (ii) $\Rightarrow $ (i).
	Define $K \coloneqq \Upsilon_0 \Xi_0^{\dagger} 
	\in \mathcal{L}(X,U)$.
	By Lemma~\ref{lem:Xi1_rep},
	we obtain $\Xi_1 \Xi_0^{\dagger} = 
	A + BK$ for all
	$(A,B) \in \Sigma_{\is}$.
	Therefore,
	$(x_1,x_0,u_0)$ are informative for stabilization with decay rate $\gamma$.
	Since
	$
	G =
	H \Xi_0^{\dagger}$,
	Lemma~\ref{lem:Douglas} (Douglas' lemma) shows that 
	there exists  $c\geq 0$ such that
	\begin{equation}
		\label{eq:im_cond_IK}
		\left\| 
		G^*
		\psi
		\right\|
		\leq c
		\left\| 
		H^*
		\psi 
		\right\|	\quad \text{for all~$\psi \in X \times U$}.
	\end{equation}
		Let
		$w \in \ell^2(\mathbb{N})$ and
		$\psi  = Hw$.
		Then
		we obtain
		$
		\left\|
		G^*H w
		\right\|
		\leq c
		\left\|
		H^*H w
		\right\|,
		$
		which is equivalent to
		the inequality \eqref{eq:operator_ineq}.
		
		Next, we prove the implication (i) $\Rightarrow $ (ii).
		To show that there exists $\Xi  \in \mathcal{L}(X,\ell^2(\mathbb{N}))$
		such that 
		$
		G =
		H \Xi$,
		it suffices by Lemma~\ref{lem:Douglas} (Douglas' lemma) to prove that
		there exists $c \geq 0$ such that 
		the inequality \eqref{eq:im_cond_IK} holds.
		By the assumption~\eqref{eq:operator_ineq},
		we obtain
		$
		\left\| 
		G^*
		\psi
		\right\|
		\leq c
		\left\| 
		H^*
		\psi
		\right\|$
		for all~$\psi \in \Ran H$.
		Since 
		$
		\left(\Ker 
		H^*\right)^{\perp} = \overline{\Ran H},
		$
		it follows that
		\begin{equation}
			\label{eq:ker_perp_ineq}
			\left\| 
			G^*
			\psi
			\right\|
			\leq c
			\left\| 
			H^*
			\psi
			\right\|	\qquad \text{for all~$\psi \in \left(\Ker 
				H^* \right)^{\perp}$}.
		\end{equation}
		As seen in Lemma~\ref{lem:ker_lemma}, we also have
		\begin{equation}
			\label{eq:ker_XiUp}
			\Ker 
			H^*
			\subseteq 
			\Ker 
			G^*.
		\end{equation}
		Combining 
		\eqref{eq:ker_perp_ineq} and 
		\eqref{eq:ker_XiUp},
		we derive
		\eqref{eq:im_cond_IK}.
		
		Since 
		$
		G =
		H \Xi$, we obtain 
		$\Xi_0 \Xi = I$.
		Therefore, $x_0$ is a frame for $X$ by
		\cite[Theorem~5.5.1]{Christensen2016}, and
		$\Xi$ is a right inverse of $\Xi_0$.
		It also follows
		from $
		G =
		H \Xi$ that
		$K = \Upsilon_0\Xi$, which together with
		Lemma~\ref{lem:Xi1_rep} yields
		$
		\Xi_1 \Xi=
		A + BK
		$
		for all $(A,B) \in \Sigma_{\is}$.
		Since $(x_1,x_0,u_0)$
		are informative for stabilization with decay rate $\gamma$,
		power stability of $\Xi_1 \Xi$ with
		decay rate $\gamma$ is obtained.
	\end{proof}
	
	The inequality \eqref{eq:operator_ineq} 
	in Theorem~\ref{thm:equiv_frame1}
	includes a stabilizing gain $K$ and hence
	cannot be checked directly by using the data.
	When $\dim U < \infty$, we can replace it
	with an inequality involving only the synthesis operator $\Xi_0$.
	\begin{theorem}
		\label{thm:equiv_frame2}
		Suppose that  
		the data $(x_1,x_0,u_0)$
		satisfy Assumptions~\ref{assump:x1x0u0_relation} and 
		\ref{assump:Bessel}. Let
		$(\Xi_1,\Xi_0,\Upsilon_0)$ be the synthesis operators 
		associated with $(x_1,x_0,u_0)$.
		If $\dim U < \infty$,
		then the following statements are equivalent
		for a fixed $\gamma \in (0,1)$:
		\begin{enumerate}[label=\upshape(\roman*), leftmargin=*, widest=ii]
			\item 
			The data $(x_1,x_0,u_0)$ are
			informative for stabilization with decay rate $\gamma$, and 
			there exists a constant $c\geq 0$ such that 
			the following inequality on  $
			\ell^2(\mathbb{N})$ holds:
			\begin{equation}
				\label{eq:operator_ineq_finite_U}
				\Xi_0 ^* \Xi_0
				\leq c^2 
				(\Xi_0^* \Xi_0)^2.
			\end{equation}
			
			\item The sequence $x_0$
			is a frame for $X$, and there exists a right inverse $\Xi_0^{\dagger} \in \mathcal{L}(X,
			\ell^2(\mathbb{N}))$ of $\Xi_0$ such that 
			$\Xi_1 \Xi_0^{\dagger}$ is power stable
			with decay rate $\gamma$.
		\end{enumerate}
	\end{theorem}
	
	\begin{proof}
		We denote by $X_0$ the closed linear span of 
		$\{ x_0(k):k \in \mathbb{N} \}$. 
		First, we prove the implication (ii) $\Rightarrow $ (i).
		Informativity for stabilization is obtained
		as in the proof of Theorem~\ref{thm:equiv_frame1}.
		Moreover, since $x_0$
		is a frame for $X$, there exists $\beta > 0$ such that
		\begin{equation}
			\label{eq:Xi0*x_lower_bound}
			\beta \|x\| \leq \|\Xi_0^*x\|
		\end{equation}
		for all $x \in X$. Let
		$w \in \ell^2(\mathbb{N})$ and
		substitute 
		$x =\Xi_0w$ into \eqref{eq:Xi0*x_lower_bound}. Then we obtain
		the inequality \eqref{eq:operator_ineq_finite_U}
		with $c = 1/\beta$ holds.
		
		Next, we prove the implication (i) $\Rightarrow $ (ii).
		By \eqref{eq:operator_ineq_finite_U}, we have
		\[
		\|\Xi_0 w\| \leq c \|\Xi_0^*\Xi_0 w\|
		\]
		for all $w \in \ell^2(\mathbb{N})$.
		Since $\overline{\Ran \Xi_0} = X_0$,
		this inequality yields
		$
		\|x\| \leq c\|\Xi_0^*x\|
		$
		for all $x \in X_0$. Lemma~\ref{lem:Douglas} (Douglas' lemma) shows that 
		$X_0 \subseteq \Ran \Xi_0$.
		Then $\Ran \Xi_0 = X_0$, and hence
		$\Ran \Xi_0$ is closed.

		Let $G  \in \mathcal{L}(X, X \times U) $ and 
		$H \in \mathcal{L}(\ell^2(\mathbb{N}), X \times U) $
		be defined as in \eqref{eq:GH_def}.
		By Lemma~\ref{lem:ker_lemma},
		we obtain $\Ran 
		G
		\subseteq 
		\overline{\Ran 
			H}$.
		We will prove that if $\dim U < \infty$, then
		$
		\Ran 
		G
		\subseteq 
		\Ran 
		H.
		$
		It is enough to show that 
		$
		\Ran H
		$
		is closed.
		By the closed range theorem
		(see, e.g., \cite[p.~205]{Yosida1980}),
		this property holds 
		if and only if $
		\Ran H^*
		$
		is closed.
		To prove that the adjoint
		$H^* = \begin{bmatrix}
			\Xi_0^* & \Upsilon_0^*
		\end{bmatrix}$
		has closed range,
		we use the closed range theorem again and then
		obtain that
		$\Ran \Xi_0^*$ is closed.
		This yields
		\[
		\left((\Ran \Xi_0^*)^{\perp} \right)^{\perp} = \Ran \Xi_0^*.
		\]		
		Let the sequence $(w_n)_{n \in \mathbb{N}}$ in
		$\Ran H^*$ converge to some $w \in \ell^2(\mathbb{N})$.
		For all $n \in \mathbb{N}$,
		there exist
		$w_{1,n} \in \Ran \Xi_0^*$
		and $w_{2,n} \in \Ran \Upsilon_0^*$
		such that $w_n = w_{1,n} + w_{2,n}$.
		Let $\Pi \in \mathcal{L}(\ell^2(\mathbb{N}))$
		be the orthogonal projection onto $(\Ran \Xi_0^*)^{\perp}$.
		Then
		\[
		\lim_{n \to \infty} \Pi w_{2,n} = 
		\lim_{n \to \infty} \Pi w_n
		= \Pi w.
		\]
		Since 
		$\dim U < \infty$ yields
		$\dim 
		\Pi (\Ran \Upsilon_0^*) < \infty$,
		it follows that 
		$\Pi (\Ran \Upsilon_0^*)$ is closed. 
		Hence, $\Pi w \in \Pi (\Ran \Upsilon_0^*)$. This 
		implies that
		$\Pi(w -w_2) = 0$ for some $w_2 \in \Ran \Upsilon_0^*$.
		Since there exists $w_1 \in \Ran \Xi_0^*$
		such that $w  = w_1+w_2$, we conclude that
		$
		\Ran H^*
		$
		is closed.

		By 		$
		\Ran 
		G
		\subseteq 
		\Ran 
		H
		$ and 
		Lemma~\ref{lem:Douglas} (Douglas' lemma),
		there exists $\Xi_0^{\dagger} \in \mathcal{L}(X,\ell^2(\mathbb{N}))$
		such that $G = H \Xi_0^{\dagger}$.
		Statement~(ii) follows as in the last paragraph
		of the proof of Theorem~\ref{thm:equiv_frame1}.
	\end{proof}
	
	\begin{remark}[Inequality~\eqref{eq:operator_ineq_finite_U}
		and frame for $X_0$]
		If \eqref{eq:operator_ineq_finite_U} holds,
		then $\Ran \Xi_0$ is closed, as shown in the proof of 
		the implication (i) $\Rightarrow $ (ii) of
		Theorem~\ref{thm:equiv_frame2}.
		The converse implication follows immediately
		from Lemma~\ref{lem:Douglas} (Douglas' lemma).
		Therefore,
		\cite[Corollary~5.5.2]{Christensen2016} shows that \eqref{eq:operator_ineq_finite_U} holds
		if and only if $x_0$ is a frame for $X_0$.
		\hspace*{\fill} $\triangle$ 
	\end{remark}
	
	The following result also characterizes the frame-based condition (statement~(i)), which is an extension of 
	the finite-dimensional case
	\cite[Theorem~17]{Waarde2020}.
	\begin{proposition}
		\label{prop:equiv_frame3}
		Suppose that  
		the data $(x_1,x_0,u_0)$
		satisfy Assumption~\ref{assump:Bessel}, and let 
		$(\Xi_1,\Xi_0,\Upsilon_0)$ be the synthesis operators 
		associated with $(x_1,x_0,u_0)$.
		Then the following statements are equivalent
		for a fixed $\gamma \in (0,1)$:
		\begin{enumerate}[label=\upshape(\roman*), leftmargin=*, widest=ii]
			\item 
			The sequence $x_0$
			is a frame for $X$, and there exists 
			a right inverse $\Xi_0^{\dagger} \in \mathcal{L}(X,
			\ell^2(\mathbb{N}))$ of $\Xi_0$ such that 
			$\Xi_1 \Xi_0^{\dagger}$
			is power stable with some decay rate 
			in $(0,\gamma)$.
			
			\item There exists an operator $\Lambda \in \mathcal{L}(X,\ell^2(\mathbb{N}))$
			such that $\Xi_0 \Lambda $ is self-adjoint and 
			the following inequality on $X \times X$ holds:
			\[
			\begin{bmatrix}
				\gamma^2 \Xi_0 \Lambda - I & \Xi_1\Lambda \\
				(\Xi_1\Lambda )^* &  \Xi_0 \Lambda 
			\end{bmatrix} \geq 0 .
			\]
		\end{enumerate}
	\end{proposition}
	\begin{proof}
		Suppose that statement~(ii) holds.
		By assumption, $\Xi_0 \Lambda \geq  (1/\gamma^2) I$.
		Then
		$\Xi_0 \Lambda$ is invertible in 
		$\mathcal{L}(X)$; see, e.g., \cite[Lemma A.3.83]{Curtain2020}.
		Since $\Xi_0$ is surjective,
		$x_0$
		is a frame for $X$ by \cite[Theorem~5.5.1]{Christensen2016}.
		The Schur-complement argument holds for bounded operators by the same proof as in the matrix case \cite[Section~A.5.5]{Boyd2004}.
		Therefore,
		the rest of the proof follows the approach
		based on Schur complements and Lyapunov inequalities
		as in the finite-dimensional case \cite[Theorem~17]{Waarde2020};
		see, e.g., 
		\cite[Theorem~II.6.1]{Eisner2010} and
		\cite[Exercise 4.22]{Curtain2020}
		for Lyapunov-based stability analysis 
		of discrete-time infinite-dimensional systems.
	\end{proof}

	\section{Informativity of finite-length data with partial system knowledge}
	\label{sec:finite_data}
	The result obtained so far requires data of infinite length.
	In this section, we address stabilization with finite-length data under the assumption that
	some partial information on the true system is already known. 
	
	\subsection{Assumption on system structures}
	Let $N \in \mathbb{N}$. Let
	$x_1 = (x_1(k))_{k=1}^N$ and
	$x_0 = (x_0(k))_{k=1}^N$ be finite sequences in $X$. Let
	$u_0 = (u_0(k))_{k=1}^N$ be a finite sequence in $U$.
	We call $(x_1,x_0,u_0)$ {\em finite data (of length $N$)}.
	Instead of Assumption~\ref{assump:x1x0u0_relation},
	here we assume that 
	the data are generated   from the true system $(A_s,B_s)$ on a finite horizon.
	\begin{assumption}
		\label{assump:x1x0u0_relation_finite}
		The finite data $(x_1,x_0,u_0)$ satisfy
		\begin{equation}
			\label{eq:x1x0u0_relation_finite}
			x_1(k) = A_sx_0(k) + B_s u_0(k)
			\quad \text{for all $k =1,\dots,N$}.
		\end{equation}
	\end{assumption}

	We make the following assumption on the operator $A_s$ of the true system  $(A_s,B_s)$.
	\begin{assumption}
		\label{assump:system_decomposition}
		Let $\gamma_- \in (0,1)$. Let the subspaces $X_+$ and $X_-$ 
		of $X$ 
		satisfy $X = X_+ \dotplus X_-$
		and $\dim X_+ < \infty$. The operator $A_s \in \mathcal{L}(X)$ satisfies
		$A_s  X_- \subseteq X_-$ and 
		$\sr(A_s|_{X_-}) \leq\gamma_-$.
	\end{assumption}
	This assumption is satisfied
	for some constant $\gamma_-\in(0,1)$ and some
	subspaces $X_+$ and $X_-$, e.g., when the true system $(A_s,B_s)$
	is stabilizable and $B_s$ is compact; see
	\cite[Theorem~4]{Logemann1992}.
	In this section, we suppose that 
	the partial information on $A_s$ described in
	Assumption~\ref{assump:system_decomposition} is given 
	for a known constant $\gamma_- \in (0,1)$ and known
	subspaces $X_+$ and $X_-$.
	For such $A_s$,
	it is enough to consider the subset $\mathcal{A}(X)$ of $\mathcal{L}(X)$
	defined
	by
	\[
	\mathcal{A}(X) \coloneqq 
	\{
	A \in \mathcal{L}(X): 
	\text{$A  X_- \subseteq X_-$ and $\sr(A|_{X_-}) \leq \gamma_-$}
	\}.
	\]
	Let $\Pi \in \mathcal{L}(X)$ be
	the projection onto $X_+$ along $X_-$, i.e., $\Pi x \coloneqq x_+$ for $x = x_++x_- \in X$ with $x_+ \in X_+$ and $x_- \in X_-$.
	Using $X=X_+ \dotplus X_-$, we decompose
	$A \in \mathcal{A}(X)$ and $B \in \mathcal{L}(U,X)$ as
	\begin{equation}
		\label{eq:AB_decomposition}
		A = 
		\begin{bmatrix}
			A_+ & 0 \\
			A_{\pm} & A_-
		\end{bmatrix}\quad \text{and}\quad 
		B = 
		\begin{bmatrix}
			B_+ \\ B_-
		\end{bmatrix},
	\end{equation}
	where 
	$A_+ \coloneqq \Pi A|_{X_+}$,
	$A_{\pm} \coloneqq (I-\Pi) A|_{X_+}$, 
	$A_- \coloneqq (I-\Pi) A|_{X_-}$,
	$B_+ \coloneqq \Pi B$, and
	$B_- \coloneqq (I-\Pi)B$.
	
	Given finite data $(x_1,x_0,u_0)$ of length $N$, we define
	the set $\Sigma_{\is,\mathcal{A}}$ of systems by
	\[
	\Sigma_{\is,\mathcal{A}} \coloneqq  \{
	(A,B) \in \mathcal{A}(X) \times \mathcal{L}(U,X):
	x_1(k) = Ax_0(k) + B u_0(k)~\text{for all $k=1,\dots,N$} \}.
	\]
	All systems in $\Sigma_{\is,\mathcal{A}}$ are compatible with
	the finite data $(x_1,x_0,u_0)$ and the prior knowledge of the true system $(A_s,B_s)$
	given in Assumption~\ref{assump:system_decomposition}.
	Under Assumptions~\ref{assump:x1x0u0_relation_finite} and \ref{assump:system_decomposition}, we have
	$(A_s,B_s) \in \Sigma_{\is,\mathcal{A}}$.
	By
	extending the sequences
	by zeros for $k \geq N+1$,
	we define the synthesis operators $\Xi_1, \Xi_0\in \mathcal{L}(\mathbb{C}^N,X)$ and $\Upsilon_0\in \mathcal{L}(\mathbb{C}^N,U)$ associated with $x_1$,
	$x_0$, and $u_0$, respectively,
	in the same way as in \eqref{eq:synthesis}.
	Let $\widetilde \Pi \in \mathcal{L}(X,X_+)$ be defined by
	$\widetilde \Pi x \coloneqq \Pi x$ for $x \in X$.
	Define 
	$\Xi_{1+},\Xi_{0+} \in \mathcal{L}(\mathbb{C}^N,X_+)$ 
	by
	\begin{equation}
		\label{eq:Xi+_def}
		\Xi_{1+} \coloneqq
		\widetilde \Pi \Xi_1\quad \text{and} \quad 
		\Xi_{0+} \coloneqq \widetilde \Pi \Xi_0.
	\end{equation}
	Since $\dim X_+ < \infty$, we can regard
	$\Xi_{1+}$ and $\Xi_{0+}$  
	as matrices.
	We will use $\widetilde\Pi$ to construct 
	a right inverse of $\Xi_{0+}$.
	We give an example of a partially unknown system that 
	can be handled by using the set $\Sigma_{\is,\mathcal{A}}$.
	\begin{example}
		\label{ex:heat_ODE}
		Consider the following 
		cascade system of the heat equation
		and  a finite-dimensional system:
		\begin{equation*}
			\left\{
			\begin{alignedat}{2}
				v(k+1) &= A_v v(k) + B_vu(k),\quad k \in \mathbb{N}_0;\qquad v(0)=v^0 \in \mathbb{C}^m, \\
				\frac{\partial \zeta}{\partial t}(\xi,t) &= 
				a\frac{\partial^2 \zeta}{\partial \xi^2}(\xi,t) + b\zeta(\xi,t),\quad 
				t \geq 0 
				;\qquad 
				&&
				\hspace{-95pt}
				\zeta(\xi,0) = \zeta^0(\xi), \\
				\frac{\partial \zeta}{\partial \xi}(0,t) &= -C_v v(k),
				\quad \frac{\partial \zeta}{\partial \xi}(1,t) = 0,\quad 
				k\tau \leq t < (k+1)\tau,~k \in \mathbb{N}_0,
			\end{alignedat}
			\right.
		\end{equation*}
		where 
		$A_v \in \mathbb{C}^{m \times m}$, $B_v \in \mathbb{C}^{m \times p}$, and $C_v \in \mathbb{C}^{1 \times m}$
		are unknown matrices. We also assume that the constants
		$a>0$ and $b \in \mathbb{R}$ are unknown but that 
		the bounds $a_0>0$ and $b_0 \in \mathbb{R}$
		satisfying
		$a \geq a_0$ and $b \leq b_0$ are already obtained.
		The sampling period $\tau >0$ is known, and we 
		measure $v(k)$ and $\zeta(\xi,k\tau)$, $\xi \in (0,1)$,
		at each $k \in \mathbb{N}_0$.
		
		The eigenvalues $(\lambda_n)_{n \in \mathbb{N}_0}$ of 
		the generator of the $C_0$-semigroup on $Z \coloneqq 
		L^2(0,1)$ governing
		the state evolution of the uncontrolled heat equation
		are 
		$\lambda_n = \lambda_n(a,b) = -a\pi^2 n^2 + b$ for $n \in \mathbb{N}_0$.
		Let $\phi_n \in Z$ be the eigenvector
		corresponding to $\lambda_n$
		for $n \in \mathbb{N}_0$. Then
		$\phi_0(\xi) = 1$ and 
		$\phi_n(\xi) = \sqrt{2} \cos(n\pi \xi )$ for 
		$\xi \in (0,1)$ and
		$n \in \mathbb{N}$.
		Note that the eigenvectors do not depend on the 
		unknown parameters $a$ and $b$.
		Define 
		$A_{vz} \in \mathcal{L}(\mathbb{C}^m, Z)$
		and $A_z \in \mathcal{L}(Z)$ by
		\begin{align*}
			A_{vz} v &\coloneqq 
			C_vv \sum_{n=0}^{\infty} 
			\left(
			\int_0^{\tau} e^{\lambda_n t} dt
			\right) \phi_n(0) \phi_n,\quad v \in \mathbb{C}^m, \\
			A_z z &\coloneqq 
			\sum_{n=0}^{\infty} e^{\lambda_n \tau }
			\langle z,\phi_n \rangle \phi_n,\quad z \in L^2(0,1).
		\end{align*}
		Applying the argument in Example~\ref{ex:sampled_data}
		to the heat equation,
		we obtain the following discrete-time system
		with state space $X\coloneqq \mathbb{C}^m \times Z$ and input space $U \coloneqq \mathbb{C}^p$:
		\begin{equation*}
			\begin{bmatrix}
				v(k+1) \\ z(k+1)
			\end{bmatrix}
			= 
			\begin{bmatrix}
				A_v & 0 \\
				A_{vz} & A_z 
			\end{bmatrix}
			\begin{bmatrix}
				v(k) \\ z(k)
			\end{bmatrix} + 
			\begin{bmatrix}
				B_v \\ 0
			\end{bmatrix}
			u(k),\quad 
			k \in \mathbb{N}_0,
		\end{equation*}
		where $z(k) \coloneqq \zeta(\cdot ,k\tau)$ for
		$k \in \mathbb{N}_0$.

		Let $\gamma_- \in (0,1)$,
		and take $n_0
		\in \mathbb{N}_0$ such that
		$e^{\lambda_{n_0}(a_0,b_0)\tau} \leq \gamma_-$, i.e., 
		\begin{equation}
			\label{eq:n0_cond}
			n_0^2 \geq
			\frac{\log(1/\gamma_- ) + b_0\tau}{a_0\pi^2 \tau}.
		\end{equation}
		We decompose $Z$ as $Z = Z_+ \dotplus Z_-$,
		where
		$Z_{+}$ is the linear span of $\{\phi_n: 
		n < n_0 \}$ and 
		$Z_{-}$ is the closed linear span of $\{\phi_n:
		n \geq n_0  \}$. Then we obtain
		$\dim Z_{+} = n_0 < \infty$ and
		$A_zZ_{-} \subseteq Z_{-}$.
		Moreover,
		$\sr(A_{z}|_{Z_{-}}) \leq \gamma_- $ by \eqref{eq:n0_cond}.
		Therefore, the true system 
		satisfies Assumption~\ref{assump:system_decomposition}
		with $X_+ \coloneqq \mathbb{C}^m \times Z_{+}$
		and $X_- \coloneqq Z_{-}$.
		\hspace*{\fill} $\triangle$ 
		\end{example}
			
			\subsection{\texorpdfstring{Informativity for stabilization on $\bm{X_+}$}{Informativity for stabilization on X\_+}}
			If the feedback gain
			$K \in \mathcal{L}(X,U)$ satisfies $K|_{X_-} = 0$, i.e., is 
			of the form
			\begin{equation}
				\label{eq:K+}
				K = K_+\Pi,\quad 
				\text{where $K_+ \coloneqq K|_{X_+}$},
			\end{equation}
			then the system $(A,B) \in \mathcal{A}(X) \times \mathcal{L}(U,X)$
			written as \eqref{eq:AB_decomposition} satisfies
			\begin{equation}
				\label{eq:ABK_rep}
				A+BK = \begin{bmatrix}
					A_++B_+K_+ & 0 \\
					A_{\pm} +B_-K_+ & A_-
				\end{bmatrix}.
			\end{equation}
			Recall that $\sr(A|_{X_-}) \leq \gamma_-$.
			For a fixed $\gamma \geq \gamma_-$,
			we have that 
			$\sr (A+BK) \leq \gamma$ is equivalent to $\sr(A_++B_+K_+) \leq \gamma$.
			Throughout this section, we focus on feedback gains of the form 
			\eqref{eq:K+} and define the subspace $\mathcal{K}(X,U)$ of $\mathcal{L}(X,U)$
			by
			\[
			\mathcal{K}(X,U) \coloneqq \{
			K \in \mathcal{L}(X,U): K|_{X_-} = 0
			\}.
			\]
			For $K \in \mathcal{K}(X,U)$ and $\gamma \geq \gamma_-$, 
			we define  the set $\Sigma_{K,\gamma,\mathcal{A}}$ of systems by
			\begin{align*}
				\Sigma_{K,\gamma,\mathcal{A}} \coloneqq \{
				(A,B) \in \mathcal{A}(X) \times \mathcal{L}(U,X):
				&~\text{There exists $M\geq 1$ such that} \\
				&~\text{$\|(A+BK)^k\| \leq M \gamma^k$ for all $k \in \mathbb{N}_0$}
				\}.
			\end{align*}

			We introduce a notion 
			of informativity for stabilization tailored to the system sets
			$\Sigma_{\is,\mathcal{A}}$ and $\Sigma_{K,\gamma,\mathcal{A}}$.
			\begin{definition}
				Suppose that Assumption~\ref{assump:system_decomposition} holds,
				and let $\gamma \in [\gamma_-,1)$.
				The finite data $(x_1,x_0,u_0)$ are called
				{\em informative for stabilization on $X_+$ with decay rate $\gamma$} if there exist an operator 
				$K \in \mathcal{K}(X,U)$ 
				such that $\Sigma_{\is,\mathcal{A}} \subseteq \Sigma_{K,\gamma,\mathcal{A}}$.
				An operator $K\in \mathcal{L}(X,U)$ is called
				a {\em stabilizing gain on $X_+$ with decay rate $\gamma$  for $(x_1,x_0,u_0)$} if
				$\Sigma_{\is,\mathcal{A}} \subseteq \Sigma_{K,\gamma,\mathcal{A}}$.
			\end{definition}
			
			The following result gives a necessary and sufficient condition
			for
			informativity for stabilization on $X_+$,
			which can be checked by using only matrices.
			\begin{theorem}
				\label{thm:informative_stabilization_finite}
				Suppose that Assumptions~\ref{assump:x1x0u0_relation_finite} and \ref{assump:system_decomposition} hold
				for the
				finite data $(x_1,x_0,u_0)$
				of length $N \in \mathbb{N}$ and the true system $(A_s,B_s)$.
				Let $(\Xi_1,\Xi_0,\Upsilon_0)$ be the synthesis operators 
				associated with $(x_1,x_0,u_0)$, 
				and define $\Xi_{1+},\Xi_{0+}\in \mathcal{L}(\mathbb{C}^N,X_+)$ by \eqref{eq:Xi+_def}.
				Then the following statements are equivalent for a fixed $\gamma \in (\gamma_-,1)$:
				\begin{enumerate}
					[label=\upshape(\roman*), leftmargin=*, widest=ii]
					\item 
					The finite data $(x_1,x_0,u_0)$ are
					informative for stabilization on $X_+$ with
					some  decay rate in $[\gamma_-,\gamma)$.
					
					\item $\Ran \Xi_{0+} = X_+$, and there exists a right inverse $\Xi_{0+}^{\dagger} \in \mathcal{L}(X_+,
					\mathbb{C}^N)$ of $\Xi_{0+} $ such that 
					$\Xi_{1+}\Xi_{0+}^{\dagger}$ is power stable with some  decay rate in $[\gamma_-,\gamma)$.
				\end{enumerate}
			\end{theorem}
			
			To prove Theorem~\ref{thm:informative_stabilization_finite},
			the following lemma is useful.
			\begin{lemma}
				\label{lem:data_reduction}
				Suppose that Assumptions~\ref{assump:x1x0u0_relation_finite} and \ref{assump:system_decomposition} hold
				for the
				finite data $(x_1,x_0,u_0)$
				of length $N \in \mathbb{N}$ and the true system $(A_s,B_s)$.
				Let $\Pi \in \mathcal{L}(X)$ be
				the projection onto $X_+$ along $X_-$.
				Then the following statements are equivalent
				for $A_+ \in \mathcal{L}(X_+)$ and
				$B_+ \in \mathcal{L}(U,X_+)$.
				\begin{enumerate}[label=\upshape(\roman*), leftmargin=*, widest=ii]
					\item There exists $(A,B) \in \Sigma_{\is,\mathcal{A}}$ such that 
					$A_+ = \Pi A|_{X_+}$ and $B_+ = \Pi B$.
					\item 
					$
					\Pi x_1(k) = A_+ \Pi x_0(k) + B_+ u_0(k)
					$
					for all $k =1,\dots,N$.
				\end{enumerate}
			\end{lemma}
			\begin{proof}
				First, we show the implication
				(ii) $\Rightarrow$ (i).
				Using the decomposition $X = X_+ \dotplus X_-$,
				we define
				$A \in \mathcal{L}(X)$ and 
				$B \in \mathcal{L}(U,X)$ by
				\begin{equation}
					A  \coloneqq
					\begin{bmatrix}
						A_+ & 0 \\
						(I-\Pi) A_s|_{X_+} & A_s|_{X_-} 
					\end{bmatrix}\quad \text{and} \quad 
					B \coloneqq
					\begin{bmatrix}
						B_+ \\ (I - \Pi)B_s 
					\end{bmatrix}.
				\end{equation}
				Then 
				$A_+ = \Pi A|_{X_+}$ and $B_+ = \Pi B$.
				Moreover,
				$A \in \mathcal{A}(X)$
				by Assumption~\ref{assump:system_decomposition}.
				Assumption~\ref{assump:x1x0u0_relation_finite}
				implies that
				\[
				(I-\Pi) x_1(k) =
				(I-\Pi) A_s\Pi x_{0}(k) 
				+ A_s(I- \Pi) x_0(k) + (I - \Pi)B_s u_0(k)
				\]
				for all $k = 1,\dots,N$.
				This together with statement~(ii) implies that
				\begin{align*}
					\begin{bmatrix}
						\Pi x_1(k) \\
						(I-\Pi) x_1(k)
					\end{bmatrix} = 
					\begin{bmatrix}
						A_+ & 0 \\
						(I-\Pi) A_s|_{X_+} & A_s|_{X_-} 
					\end{bmatrix}
					\begin{bmatrix}
						\Pi x_0(k) \\ (I- \Pi) x_0(k)
					\end{bmatrix} + 
					\begin{bmatrix}
						B_1 \\ (I - \Pi)B_s 
					\end{bmatrix} u_0(k)
				\end{align*}
				for all $k = 1,\dots,N$.
				Hence, $(A,B) \in \Sigma_{\is,\mathcal{A}}$.
				
				Next, we prove the implication
				(i) $\Rightarrow$ (ii).
				By
				$(A,B) \in \Sigma_{\is,\mathcal{A}}$,
				we obtain
				\begin{equation}
					\label{eq:Pi_x1_x0_u0}
					\Pi x_1(k) = \Pi Ax_0(k) + B_1 u_0(k)
				\end{equation}
				for all $k = 1,\dots,N$.
				Since $AX_- \subseteq X_-$,
				it follows that $\Pi A = A_+ \Pi$.
				Substituting this into \eqref{eq:Pi_x1_x0_u0}, 
				we conclude that statement (ii) holds.
			\end{proof}
			
			Using Lemma~\ref{lem:data_reduction}, we
			obtain a range property 
			of synthesis operators as in  \cite[Lemma~15]{Waarde2020} and
			Lemma~\ref{lem:ker_lemma}.
			\begin{lemma}
				\label{lem:ker_lemma_finite}
				Suppose that Assumptions~\ref{assump:x1x0u0_relation_finite} and \ref{assump:system_decomposition} hold
				for the
				finite data $(x_1,x_0,u_0)$
				of length $N \in \mathbb{N}$ and the true system $(A_s,B_s)$.
				Let $(\Xi_1,\Xi_0,\Upsilon_0)$ be the synthesis operators 
				associated with $(x_1,x_0,u_0)$.
				Assume that
				$K \in \mathcal{K}(X,U)$ satisfies
				$\Sigma_{\is,\mathcal{A}} \subseteq
				\Sigma_{K,\gamma, \mathcal{A}}$
				for some $\gamma \geq \gamma_-$, and let 
				$K_+ \coloneqq K|_{X_+}$. Then
				\begin{equation}
					\label{eq:Xi_Up_ran_inclusion_finite}
					\Ran \begin{bmatrix}
						I \\ K_+
					\end{bmatrix}
					\subseteq
					\Ran \begin{bmatrix}
						\Xi_{0+} \\ \Upsilon_0
					\end{bmatrix},
				\end{equation}
				where $\Xi_{0+}\in \mathcal{L}(\mathbb{C}^N,X_+)$ is as in \eqref{eq:Xi+_def}.
			\end{lemma}
			\begin{proof}
				Let $\Pi \in \mathcal{L}(X)$ be
				the projection onto $X_+$ along $X_-$.
				For the finite data $(\Pi x_{1}, \Pi x_{0}, u_0)$ of
				length $N$,
				define the subsets $\Sigma_{\is+}$  and $\Sigma_{\is+}^0$ of $ \mathcal{L}(X_+) \times \mathcal{L}(U,X_+)$
				by
				\begin{align*}
					\Sigma_{\is+} &\coloneqq 
					\{
					(A_+,B_+):
					\Pi x_1(k) = A_+ \Pi x_{0}(k) + B_+ u_0(k)~\text{for all
						$k =1,\dots,N$}
					\}, \\
					\Sigma_{\is+}^0 &\coloneqq 
					\{
					(A_+,B_+):
					0 = A_+ \Pi x_{0}(k) + B_+ u_0(k)~\text{for all
						$k =1,\dots,N$}
					\}.
				\end{align*}
				By assumption,
				$\sr(A + BK) \leq \gamma$
				for all $(A,B) \in \Sigma_{\is,\mathcal{A}}$.
				By Lemma~\ref{lem:data_reduction},
				for all $(A_+,B_+) \in \Sigma_{\is+}$,
				there exists $(A,B) \in \Sigma_{\is,\mathcal{A}}$ such that 
				$A_+ = \Pi A|_{X_+}$ and $B_+ = \Pi B$.
				Therefore, \eqref{eq:ABK_rep} implies that
				$\sr(A_++B_+K_+) \leq \gamma$
				for all $(A_+,B_+) \in \Sigma_{\is+}$.

				Let $(A_+,B_+)
				\in \Sigma_{\is+}$ and $(A_{0},B_{0})
				\in \Sigma_{\is+}^0$.
				Define $F \coloneqq A_++B_+K_+$ and $F_0 \coloneqq A_{0}+B_{0}K_+$.
				Then 
				for all $n \in \mathbb{N}$,
				$(A_++nA_{0}, B_++nB_{0}) \in \Sigma_{\is+}$
				and hence $\sr(F+nF_0)\leq \gamma$.
				Since $F+nF_0 \in \mathcal{L}(X_+)$ and $\dim X_+ < \infty$, 
				the continuity of the spectral radius 
				for matrices implies that 
				$\sr(F_0) = 0$.
				Therefore, the same argument as in the proof of 
				Lemma~\ref{lem:ker_lemma} can be applied to show that
				\[
				\Ran \begin{bmatrix}
					I \\ K_+
				\end{bmatrix} \subseteq
				\overline{\Ran \begin{bmatrix}
						\Xi_{0+} \\ \Upsilon_0
				\end{bmatrix}}.
				\]
				Since
				$\Xi_{0+} \in \mathcal{L}(\mathbb{C}^N,X_+)$ and 
				$\Upsilon_0 \in \mathcal{L}(\mathbb{C}^N,U)$ are 
				finite-rank operators, it follows that
				\[
				\Ran \begin{bmatrix}
					\Xi_{0+} \\ \Upsilon_0
				\end{bmatrix}
				\]
				is closed.
				Thus, \eqref{eq:Xi_Up_ran_inclusion_finite} holds.
			\end{proof}

			Now we are in a position to prove 
			Theorem~\ref{thm:informative_stabilization_finite}.
			\begin{proof}[Proof of Theorem~\ref{thm:informative_stabilization_finite}]
				Let $\Pi \in \mathcal{L}(X)$ be
				the projection onto $X_+$ along $X_-$.
				First, we prove the implication
				(ii) $\Rightarrow$ (i).
				Define $K_+ \coloneqq \Upsilon_0 \Xi_{0+}^{\dagger}$, and 
				let $(A,B) \in \Sigma_{\is,\mathcal{A}}$. Then
				we can show that $\Xi_1 = A\Xi_0+B\Upsilon_0$, as in
				Lemma~\ref{lem:Xi1_rep}. 
				Hence,
				\begin{equation}
					\label{eq:Xi1+Xi0+_a}
					\Xi_{1+}\Xi_{0+}^{\dagger} = 
					\Pi A\Xi_0 \Xi_{0+}^{\dagger} + 
					\Pi BK_+.
				\end{equation}
				Define $A_+ \coloneqq \Pi A|_{X_+}$
				and $B_+ \coloneqq \Pi B$. Since 
				$AX_- \subseteq X_-$, it follows that
				\begin{align}
					\label{eq:Xi1+Xi0+_b}
					\Pi A \Xi_0\Xi_{0+}^{\dagger} =
					A_+ \Pi \Xi_0 \Xi_{0+}^{\dagger} = A_+.
				\end{align}
				Combining \eqref{eq:Xi1+Xi0+_a}
				and \eqref{eq:Xi1+Xi0+_b}, we obtain
				$\Xi_{1+}\Xi_{0+}^{\dagger} = A_+ + B_+K_+$.
				By the assumption on the stability of 
				$\Xi_{1+} \Xi_{0+}^{\dagger}$,
				it follows that $A_+ + B_+K_+$
				is power stable with some decay rate in
				$[\gamma_-, \gamma)$.
				This and \eqref{eq:ABK_rep} imply that 
				if $K \in \mathcal{K}(X,U)$ is defined by
				$K \coloneqq K_+ \Pi $,
				then $A+BK$ is power stable with some decay rate in
				$[\gamma_-, \gamma)$.
				Hence, statement~(i) holds.
				
				Next, we prove the implication
				(i) $\Rightarrow$ (ii).
				By assumption and \eqref{eq:ABK_rep},
				there exist $K \in \mathcal{K}(X,U)$ 
				and $\gamma_+ \in [\gamma_-,\gamma)$
				such that
				$A_+ + B_+K_+$ is power stable with decay rate $\gamma_+$ for all $(A,B) \in \Sigma_{\is,\mathcal{A}}$, where $
				A_+ \coloneqq \Pi A|_{X_+}$, 
				$B_+ \coloneqq \Pi B$, and
				$K_+ \coloneqq K|_{X_+}$.
				Since \eqref{eq:Xi_Up_ran_inclusion_finite} holds
				by Lemma~\ref{lem:ker_lemma_finite},
				it follows that $\Ran \Xi_{0+} = X_+$. Moreover,
				by  \eqref{eq:Xi_Up_ran_inclusion_finite} and
				Lemma~\ref{lem:Douglas} (Douglas' lemma),
				there exists a right inverse 
				$\Xi_{0+}^{\dagger}  \in \mathcal{L}(X_+,
				\mathbb{C}^N)$ of $\Xi_{0+}$ such that 
				\[
				\begin{bmatrix}
					I \\ K_+
				\end{bmatrix} =
				\begin{bmatrix}
					\Xi_{0+} \\ \Upsilon_0
				\end{bmatrix}
				\Xi_{0+}^{\dagger}.
				\]
				By the same calculations as in
				\eqref{eq:Xi1+Xi0+_a} and 
				\eqref{eq:Xi1+Xi0+_b}, we obtain
				$\Xi_{1+}\Xi_{0+}^{\dagger} = A_++B_+K_+$ for all
				$(A,B) \in \Sigma_{\is,\mathcal{A}}$. Thus,
				$\Xi_{1+}\Xi_{0+}^{\dagger}$
				is power stable with decay rate $\gamma_+$.
			\end{proof}
			
			From the proof of Theorem~\ref{thm:informative_stabilization_finite},
			we obtain the structure of stabilizing gains
			on $X_+$.
			\begin{remark}[Structure of stabilizing gains on 
				$X_+$]
				\label{rem:K_fin_form}
				Let the assumptions of Theorem~\ref{thm:informative_stabilization_finite}
				be satisfied.
				If statement~(ii) of Theorem~\ref{thm:informative_stabilization_finite} 
				holds, then
				$
				\Upsilon_0 \Xi_{0+}^{\dagger} \Pi \in \mathcal{K}(X,U)$
				is a stabilizing gain on $X_+$ with decay rate $\gamma$  for $(x_1,x_0,u_0)$.
				Conversely, every stabilizing gain on $X_+$ with decay rate $\gamma$  for $(x_1,x_0,u_0)$ can
				be written as 
				$\Upsilon_0 \Xi_{0+}^{\dagger} \Pi$
				for some right inverse $\Xi_{0+}^{\dagger}$ 
				of $\Xi_{0+}$ satisfying
				the properties given in statement~(ii) of
				Theorem~\ref{thm:informative_stabilization_finite}.
				\hspace*{\fill} $\triangle$ 
			\end{remark}
			
			We can also characterize informativity for 
			stabilization on $X_+$ in terms of 
			operator inequalities on $X_+ \times X_+$.
			Since $\dim X_+ < \infty$, the inequalities can be
			written as linear matrix inequalities.
			Since the proof is the same as
			\cite[Theorem~17]{Waarde2020},
			it is omitted here.
			\begin{corollary}
				\label{coro:ope_ineq_finite}
				Suppose that Assumptions~\ref{assump:x1x0u0_relation_finite} and \ref{assump:system_decomposition} hold
				for the
				finite data $(x_1,x_0,u_0)$
				of length $N \in \mathbb{N}$ and the true system $(A_s,B_s)$.
				Let $(\Xi_1,\Xi_0,\Upsilon_0)$ be the synthesis operators 
				associated with $(x_1,x_0,u_0)$, 
				and let $\Xi_{1+},\Xi_{0+}\in \mathcal{L}(\mathbb{C}^N,X_+)$ be defined by \eqref{eq:Xi+_def}.
				Then the following statements are equivalent for a fixed $\gamma \in (\gamma_-,1)$:
				\begin{enumerate}
					[label=\upshape(\roman*), leftmargin=*, widest=ii]
					\item 
					The finite data $(x_1,x_0,u_0)$ are
					informative for stabilization on $X_+$  with some decay rate in $ [\gamma_-, \gamma)$.
					
					\item
					There exists
					an operator $\Lambda \in 
					\mathcal{L}(X_+,\mathbb{C}^N)$ such that 
					$\Xi_{0+} \Lambda $ is self-adjoint and 
					the following inequality on $X_+ \times X_+$ holds:
					\[
					\begin{bmatrix}
						\gamma^2 \Xi_{0+} \Lambda -I & \Xi_{1+}\Lambda \\
						(\Xi_{1+}\Lambda)^* & \Xi_{0+} \Lambda
					\end{bmatrix} \geq 0 .
					\]
				\end{enumerate}
			\end{corollary}
			
			For the structure of stabilizing gains on $X_+$,
			the same result as in Remark~\ref{rem:K_fin_form}
			holds by replacing $\Xi_{0+}^{\dagger}$ with $\Lambda (\Xi_{0+} \Lambda)^{-1}$.
			
			\begin{example}[Example~\ref{ex:heat_ODE} revisited]
				Consider the cascade system in
				Example~\ref{ex:heat_ODE} with true parameters
				\[
				A_v = 
				\begin{bmatrix}
					1 & 0.5 \\ -0.5 & 1
				\end{bmatrix},\quad 
				B_v = 
				\begin{bmatrix}
					2 \\ -1 
				\end{bmatrix},\quad 
				C_v = 
				\begin{bmatrix}
					0.5 &  1
				\end{bmatrix},\quad 
				a = 0.2,\quad \text{and} \quad b = -0.1
				\]
				and known bounds
				$a_0 = 0.1$ and $b_0 = 0$. The sampling period is $\tau = 0.05$.
				The finite data $(x_1,x_0,u_0)$ of length $N$ 
				are obtained from a single trajectory as in Example~\ref{ex:single_trajectory}, where
				the initial state and the input are chosen as
				\[
				v(0) = 
				\begin{bmatrix}
					0 \\ 0
				\end{bmatrix},\quad 
				\zeta(\xi,0) \equiv 1,\quad \text{and} \quad
				u_0(k) \equiv 1.
				\]
				Set $\gamma_- = 0.89$. Then the minimum
				$n_0 \in \mathbb{N}_0$ satisfying the inequality
				\eqref{eq:n0_cond} is $n_0 = 2$. 
				For the data length $N=5$,
				Theorem~\ref{thm:informative_stabilization_finite} and 
				Corollary~\ref{coro:ope_ineq_finite} show that 
				the data $(x_1,x_0,u_0)$ are informativity for 
				stabilization on $X_+$ with decay rate $\gamma = 0.9$.
				Corollary~\ref{coro:ope_ineq_finite} further yields a stabilizing gain 
				$K = K_+ \Pi$, where
				\[
				K_+ \coloneqq
				\begin{bmatrix}
					-0.4615 & 0.6464 & 1.6387 & -0.1597
				\end{bmatrix}.
				\] 
				All linear matrix inequalities were solved using 
				MATLAB R2025b with YALMIP~\cite{Lofberg2004} and 
				MOSEK~\cite{MOSEK2025}.
				\hspace*{\fill} $\triangle$ 
			\end{example}

			\section{Conclusion}
			\label{sec:conclusion}
			This paper investigated 
			data-driven stabilization for discrete-time
			infinite-dimensional systems.
			For single-input systems,
			we established a 
			necessary and sufficient condition for 
			infinite-length data to be 
			informative for stabilization in the noise-free setting.
			If the state sequence forms a frame and 
			the noise is small relative to the measured data, then
			this condition still implies stabilization 
			in the presence of noise.
			Combining finite-length data with 
			prior structural knowledge of the true system, 
			we derived a necessary and sufficient condition 
			for data informativity.
			Future work will examine informativity for 
			other system properties, including dissipativity and 
			admissibility.

\end{document}